\documentclass[a4paper,10pt,oneside,final,leqno,notitlepage]{article}

\usepackage{amsmath}
\usepackage{amsfonts}
\usepackage{amssymb}
\usepackage{amsthm}

\usepackage{graphicx}

\usepackage[blocks]{authblk}

\usepackage[T1]{fontenc}

\usepackage{enumerate}
\def\CC{\mathbb{C}} %C
\def\DD{\mathbb{D}} %D
\def\BB{\mathbb{B}} %B
\def\RR{\mathbb{R}} %R
\def\TT{\mathbb{T}} %T (brzeg D)
\def\SS{\mathbb{S}} %S (pas pionowy w C, lub sfera w R^n, itp.)
\def\CDD{\overline{\DD}} %domkniecie D
\def\OO{\mathcal{O}} %holom.
\def\CL{\mathcal{C}} %klasa C^k, C^r, itp.
\def\RE{\textnormal{Re}\,} %czesc rzeczywista
\def\IM{\textnormal{Im}\,} %czesc urojona
\def\INT{\textnormal{int}\,} %wnetrze

\def\SPAN{\textnormal{span}\,}
\def\AUT{\textnormal{Aut}\,}
\def\HH{\mathcal{H}} %rodzina H
\def\HHP{\HH_+} %rodzina H_+
\def\MMM{\mathcal{M}} %rodzina M

\def\LEBT{\mathcal{L}^{\TT}} %znaczek L^T, dla miary Lebesgue'a na okregu jednostkowym
\def\DLEBT{d\mathcal{L}^{\TT}} %znaczek dL^T, dla miary Lebesgue'a na okregu jednostkowym

\def\BLHL#1{\bar\lambda h_{#1}(\lambda)} %\bar\lambda h_j(\lambda)

\newtheoremstyle{remarkstyle}{}{}{}{}{\bf}{.}{ }{}

\newtheorem{THE}{Theorem}[section]
\newtheorem{PROP}[THE]{Proposition}
\newtheorem{LEM}[THE]{Lemma}
\newtheorem{CRL}[THE]{Corollary}
\newtheorem{OBS}[THE]{Observation}

\theoremstyle{remarkstyle}
\newtheorem{RM}[THE]{Remark}
\newtheorem{EX}[THE]{Example}

\begin{document}

\renewcommand{\thepage}{\small\arabic{page}}
\renewcommand{\thefootnote}{(\arabic{footnote})}

\renewcommand\Affilfont{\small}

% **** **** **** **** **** **** **** **** **** **** **** **** **** **** **** ****

%\frontmatter

\author{Sylwester Zaj\k{a}c}
\affil{Institute of Mathematics, Faculty of Mathematics and Computer Science,\\ Jagiellonian University, \L ojasiewicza 6, 30-348 Krak\'ow, Poland\\ sylwester.zajac@im.uj.edu.pl}

% \address{Institute of Mathematics, Faculty of Mathematics and Computer Science,\\ Jagiellonian University, \L ojasiewicza 6, 30-348 Krak\'ow, Poland\\ sylwester.zajac@im.uj.edu.pl}

\title{Complex geodesics in convex tube domains II}

\date{}

\maketitle

\begin{abstract}
We give a description (direct formulas) of all complex geodesics in a convex tube domain in $\CC^n$ containing no complex affine lines, expressed in terms of geometric properties of the domain.
We next apply that result to give formulas (a necessary condition) for extremal mappings with respect to the Lempert function and the Kobayashi-Royden metric in a big class of bounded, pseudoconvex, complete Reinhardt domains: for all of them in $\CC^2$ and for those of them in $\CC^n$ which logarithmic image is strictly convex in geometric sense.
\end{abstract}

\footnotetext[1]{


\noindent
The author was supported by the NCN grant on the basis of the decision number DEC-2012/05/N/ST1/02911.
}

% **** **** **** **** **** **** **** **** **** **** **** **** **** **** **** ****

%\mainmatter

% -------- -------- -------- -------- -------- --------

\section{Introduction}\label{sect_introduction}

A non-empty open set $D\subset\CC^n$ is called a \emph{tube domain} if $D=\Omega+i\RR^n$ for some domain $\Omega\subset\RR^n$.
We call $\Omega$ the \emph{base} of $D$ and in this paper we denote it by $\RE D$.
In the recent paper \cite{zajac} we investigated convex tube domains from the point of view of theory of holomorphically invariant distances.
More precisely, we were interested especially in the notion of complex geodesics.
Given a convex domain $D\subset\CC^n$, we call a holomorphic map $\varphi:\DD\to D$ a \emph{complex geodesic} for $D$ if there exists a \emph{left inverse} of $\varphi$, i.e. a holomorphic function $f:D\to\DD$ such that $f\circ\varphi=\text{id}_{\DD}$.
Complex geodesics of $D$ are exactly holomorphic isometires between the unit disc $\DD\subset\CC$ equipped with the Poincar\'{e} distance and the domain $D$ equipped with the Carath\'{e}odory pseudodistance.
It follows from the Lempert theorem (see \cite{lempert} or \cite[Chapter 8]{jarnickipflug}) that if $D\subset\CC^n$ is a taut convex domain, then for any pair of points in $D$ there exists a complex geodesic passing through them.

In this paper we restrict our considerations to convex tube domains containing no complex affine lines (equivalently, to convex tubes with no real affine lines contained in base).
Such a family of domains is equal to the family of taut convex tube domains (see e.g. \cite{braccisarcco}).
This approach has many advantages, among which it is worth mentioning that every holomorphic map with image lying in such a domain admits a boundary measure (\cite[Observation 2.5]{zajac}).
What is more, from \cite[Observation 2.4]{zajac} it follows that doing such a restriction we lose no generality.

This paper may be treated as a continuation of \cite{zajac}.
In \cite{zajac} we gave an equivalent condition for a holomorphic map $\varphi:\DD\to D$ to be a complex geodesic in a convex tube domain $D$ containing no complex affine lines.
It it stated in language of measure theory and formulated in terms of \emph{boundary $n$-tuple of measures} of $\varphi$ (in \cite{zajac} and here shortly called a \emph{boundary measure}).
Boundary measure of $\varphi$ is a unique $n$-tuple $\mu=(\mu_1,\ldots,\mu_n)$ of real Borel measures on the unit circle $\TT\subset\CC$ such that $$\varphi(\lambda)=\frac{1}{2\pi}\int_{\TT}\frac{\zeta+\lambda}{\zeta-\lambda}d\mu(\zeta)+i\IM\varphi(0),\;\lambda\in\DD.$$
In the main result of this paper, Theorem \ref{th_postac_geod}, we present a full description of all complex geodesics for $D$.
We derive it using the equivalent condition from \cite{zajac} and the following, 'spherical' decomposition of $n$-tuples of measures (Lemma \ref{lem_rozklad_miary_varrho_d_nu}): given real Borel measures $\mu_1,\ldots,\mu_n$ on $\TT$, there exist a finite positive Borel measure $\nu$ on $\TT$ singular to the Lebesgue measure $\LEBT$ on $\TT$, a Borel-measurable map $\varrho$ from $\TT$ to the unit sphere $\partial\BB_n$ and a map $g:\TT\to\RR^n$ with components in $L^1(\TT,\LEBT)$ such that $$(\mu_1,\ldots,\mu_n)=g\,d\LEBT+\varrho\,d\nu.$$
The objects $\nu$, $\varrho$ and $g$ are in some sense unique.
Theorem \ref{th_postac_geod} states that a holomorphic map with a boundary measure $\mu=(\mu_1,\ldots,\mu_n)$ is a complex geodesic for $D$ if and only if the parts $\nu$, $\varrho$ and $g$ of the decomposition of $\mu$ satisfy several geometric conditions.
So, strictly speaking, in Theorem \ref{th_postac_geod} we describe the form of every $n$-tuple of measures which define a complex geodesic for $D$.
But in fact, the complex geodesic itself can be then easily recovered (up to an imaginary constant) from its boundary measure via the above integral.

Later in the paper we apply Theorem \ref{th_postac_geod} to obtain more detailed descriptions of complex geodesics in some special classes of convex tube domains.
In Subsection \ref{subsect_domains_Dn} we do it for, among others, convex tubes $D\subset\CC^n$ with the base being bounded from above on each coordinate and satisfying the equality $\RE D+(-\infty,0]^n=\RE D$.
These domains are very useful in studying extremal mappings with respect to the Lempert function and the Kobayashi-Royden metric in bounded, pseudoconvex, complete Reinhardt domains in $\CC^n$.
We deal with this topic in Section \ref{sect_reinhardt_domains}, achieving formulas for extremal mappings in a big class of such Reinhardt domains:  for all of them in $\CC^2$ and for those of them in $\CC^n$ which logarithmic image is strictly convex in geometric sense (i.e. it is convex and its boundary contains no non-trivial segments).
Besides, in Subsection \ref{subsect_domains_in_c2} we investigate complex geodesics in convex tube domains in $\CC^2$.
The results obtained there, together with the considerations made in Subsection \ref{subsect_domains_Dn}, simplify the conditions from Theorem \ref{th_postac_geod} in two-dimensional case.

Let us briefly summarize the content of the paper.
In Section \ref{sect_preliminaries} we present the notation which is used in this paper and we recall some facts about boundary measures of holomorphic maps.
There we also prove the lemma on the decomposition of $n$-tuples of measures, which was mentioned above.
At the end of that section we define a few objects describing some geometric properties of a convex tube domain in $\CC^n$.
In Section \ref{sect_special_classes} we formulate the main result of this paper, Theorem \ref{th_postac_geod}, we present its applications in special classes of tube domains and we give some examples.
Section \ref{sect_dowod_tw_postac_geod} contains the proof of Theorem \ref{th_postac_geod} and some additional remarks.
In Section \ref{sect_reinhardt_domains} we apply results from Section \ref{sect_special_classes} to obtain formulas for extremal holomorphic mappings in some classes of Reinhardt domains in $\CC^n$.

\section{Preliminaries}\label{sect_preliminaries}

Let us begin with some notation.
The symbols $\DD$, $\TT$, $\CC_*$ denote respectively the unit disc in $\CC$, the unit circle in $\CC$ and the punctured plane, namely the set $\CC\setminus\lbrace 0\rbrace$.
By $\delta_{\lambda_0}$ we mean the Dirac delta at a point $\lambda_0\in\TT$, by $\chi_{A}$ we mean the characteristic function $\chi_A:\TT\to\lbrace 0,1\rbrace$ of a set $A\subset\TT$ and by $e_1,\ldots,e_n$ we mean the canonical basis of $\RR^n$ or $\CC^n$.
The Poincar\'{e} distance in $\DD$ is denoted by $\rho$.
By $\langle x,y\rangle$ we mean the standard inner product of vectors $x,y\in\RR^n$, by $\|\cdot\|$ we denote the euclidean norm in $\RR^n$ and by $\BB_n$ we mean the unit euclidean ball in $\RR^n$.
For a set $A\subset\RR^n$ the symbol $A^{\perp}$ denotes the set $\lbrace v\in\RR^n:\forall a\in A:\langle v,a\rangle=0\rbrace$.

We use the symbol $\langle\cdot,\cdot\cdot\rangle$ also for measures and functions.
For example, if $\mu$ is a tuple $(\mu_1,\ldots,\mu_n)$ of real (i.e. complex with real values) Borel measures on $\TT$ and $v=(v_1,\ldots,v_n)$ is a real vector or a bounded Borel-measurable mapping from $\TT$ to $\RR^n$, then $\langle d\mu,v\rangle$ or $\langle v,d\mu\rangle$ is the measure $\sum_{j=1}^n v_j d\mu_j$, etc.
% In this paper we deal only with Borel measures on $\TT$.
The fact that a real measure $\nu$ is positive (resp. negative, null) is shortly denoted by $\nu\geq 0$ (resp. $\nu\leq 0$, $\nu=0$).
The variation of a complex measure $\nu$ is denoted by $|\nu|$.
In this paper we consider mostly Borel measures on $\TT$ and we sometimes omit the word 'Borel'.

In what follows we use the following families of mappings:
\begin{equation*}
\begin{array}{lcl}
\HH^n   & := & \lbrace h\in\OO(\CC,\CC^n): \forall\lambda\in\TT:\,\BLHL{}\in\RR^n\rbrace,\\
\HHP^n & := & \lbrace h\in\OO(\CC,\CC^n): \forall\lambda\in\TT:\,\BLHL{}\in [0,\infty)^n\rbrace.
\end{array}
\end{equation*}
We have
$$\HH^n = \lbrace h\in\OO(\CC,\CC^n): \exists a\in\CC^n, b\in\RR^n: h(\lambda)=\bar a\lambda^2+b\lambda+a, \lambda\in\CC\rbrace.$$
Moreover (see e.g. \cite[Lemma 8.4.6]{jarnickipflug}),
$$\HHP^1 = \lbrace h\in\OO(\CC): \exists c\geq 0, d\in\CDD: h(\lambda)=c(\lambda-d)(1-\bar d\lambda), \lambda\in\CC\rbrace.$$
In particular, for $h\in\HHP^1$ we have $\BLHL{}=c|\lambda-d|^2$, $\lambda\in\TT$, so such a function $h$ has at most one zero on $\TT$ (counting without multiplicities).

In this paper we sometimes consider linear dependence or independence of functions $h_1,\ldots,h_m\in\HH^1$.
Note that here it does not matter whether it is meant over the filed $\RR$ or $\CC$, because these two properties are equivalent, in view of the fact that $\BLHL{j}\in\RR$ for $\lambda\in\TT$, $j=1,\ldots,m$.

\smallskip
Now we recall some facts on boundary measures of holomorphic maps.
A real Borel measure $\mu$ on $\TT$ is called \emph{boundary measure} of a holomorphic function $\varphi:\DD\to\CC$, if
\begin{equation}\label{eq_poisson_formula_f_mu}
\varphi(\lambda)=\frac{1}{2\pi}\int_{\TT}\frac{\zeta+\lambda}{\zeta-\lambda}d\mu(\zeta)+i\IM\varphi(0),\;\lambda\in\DD,
\end{equation}
or equivalently, taking the real parts in this equality, if
\begin{equation}\label{eq_poisson_formula_re_f_mu}
\RE f(\lambda)=\frac{1}{2\pi}\int_{\TT}\frac{1-|\lambda|^2}{|\zeta-\lambda|^2}d\mu(\zeta),\;\lambda\in\DD.
\end{equation}
Such a measure $\mu$ is uniquely determined by $\varphi$.
In the case when $\varphi$ is a map, namely $\varphi=(\varphi_1,\ldots,\varphi_n)\in\OO(\DD,\CC^n)$, by \emph{boundary measure} of $\varphi$ we mean a unique $n$-tuple $(\mu_1,\ldots,\mu_n)$ of real Borel measures on $\TT$ such that $\mu_j$ is the boundary measure for $\varphi_j$ for every $j=1,\ldots,n$.
Then formulas analogous to \eqref{eq_poisson_formula_f_mu} and \eqref{eq_poisson_formula_re_f_mu} hold for $\varphi$.

Denote $$\MMM^n:=\lbrace\varphi\in\OO(\DD,\CC^n):\varphi\text{ admits a boundary measure}\rbrace.$$
Not every holomorphic function on $\DD$ admits a boundary measure and hence $\mathcal{M}^n\subsetneq\OO(\DD,\CC^n)$.
It is very important that if $D\subset\CC^n$ is a convex tube domain containing no complex affine lines, then every holomorphic map $\varphi:\DD\to D$ belongs to $\MMM^n$ (see \cite[Observation 2.5]{zajac}).
In that case for $\LEBT$-almost every $\lambda\in\TT$ the radial limit $\varphi^*(\lambda)=\lim_{r\to 1^-}\varphi(r\lambda)$ of $\varphi$ exists and belongs to $\overline{D}$.

It is worth to recall that if $\mu$ is a boundry measure of a holomorphic function $\varphi\in\MMM$, then $\mu$ is a weak-* limit of measures $\RE\varphi(r\lambda)\,d\LEBT(\lambda)$, when $r\to 1^-$ (see e.g. \cite[p. 10]{koosis}).
Here we treat complex measures as linear functionals on $\CL(\TT)$, the space of all complex-valued continuous functions on $\TT$ equipped with the supremum norm.
The weak-* convergence which we mentioned means that $$\int_{\TT}u(\lambda)\RE\varphi(r\lambda)\,d\LEBT(\lambda)\xrightarrow{r\to 1^-}\int_{\TT}u(\lambda)\,d\mu(\lambda),\;u\in\CL(\TT).$$ 

We need the following fact: if $\mu$ is a boundary measure of a function $\varphi\in\MMM$ and $\mu=g\,d\LEBT+\mu_s$ is the Lebesgue-Radon-Nikodym decomposition of $\mu$ with respect to $\LEBT$, i.e. $g\in L^1(\TT,\LEBT)$ and $\mu_s$ is a real Borel measure on $\TT$ singular to $\LEBT$, then $\RE\varphi^*(\lambda)=g(\lambda)$ for $\LEBT$-a.e. $\lambda\in\TT$ (see e.g. \cite[p. 11]{koosis}).
In particular, $\RE\varphi^*\in L^1(\TT,\LEBT)$ and if $\varphi_s$ is a holomorphic function with boundary measure $\mu_s$, then $\RE\varphi_s^*(\lambda)=0$ for $\LEBT$-a.e. $\lambda\in\TT$.

In what follows, given a $n$-tuple $\mu=(\mu_1,\ldots,\mu_n)$ of real Borel measures on $\TT$, by its Lebesgue-Radon-Nikodym decomposition with respect to $\LEBT$ we mean a unique decomposition $$\mu=g\,d\LEBT+\mu_s,$$ where $g=(g_1,\ldots,g_n):\TT\to\RR^n$ is Borel-measurable, $g_1,\ldots,g_n\in L^1(\TT,\LEBT)$ and $\mu_s=(\mu_{s,1},\ldots,\mu_{s,n})$ is a $n$-tuple of real Borel measures on $\TT$, each of which is singular to $\LEBT$.
In other words, for every $j$, $$\mu_j=g_j\,d\LEBT+\mu_{s,j}$$ is the Lebesgue-Radon-Nikodym decomposition of $\mu_j$ with respect to $\LEBT$.
We call the $n$-tuples $g\,d\LEBT$ and $\mu_s$ respectively the absolutely continuous part and the singular part of $\mu$ (omitting the phrase 'in its Lebesgue-Radon-Nikodym decomposition with respect to $\LEBT$', which is assumed by default). 
The following lemma is a useful variation on the Lebesgue-Radon-Nikodym decomposition of $n$-tuples of measures:

\begin{LEM}\label{lem_rozklad_miary_varrho_d_nu}
Let $\mu$ be a $n$-tuple of real Borel measures on $\TT$. 
Then there exist a unique finite, positive Borel measure $\nu$ on $\TT$ singular to $\LEBT$, a unique (up to a set of $\nu$ measure zero) Borel-measurable map $\varrho:\TT\to\partial\BB_n$ and a unique (up to a set of $\LEBT$ measure zero) Borel-measurable map $g:\TT\to\RR^n$ with components in $L^1(\TT,\LEBT)$ such that
\begin{equation}\label{eq_lem_rozklad_miary_varrho_d_nu_postac_miary}
\mu=g\,d\LEBT+\varrho\,d\nu. 
\end{equation}
In particular, $g\,d\LEBT$ and $\varrho\,d\nu$ are (respectively) the absolutely continuous and singular parts of $\mu$ in its Lebesgue-Radon-Nikodym decomposition with respect to $\LEBT$.
\end{LEM}

Lemma \ref{lem_rozklad_miary_varrho_d_nu} follows directly from the following general fact, applied to the singular part of $\mu$:

\begin{LEM}\label{lem_rozklad_miary_varrho_d_nu_ogolniejszy}
If $(X,\mathcal{A})$ is a measurable space and $\mu=(\mu_1,\ldots,\mu_n)$ is a $n$-tuple of real measures $\mu_j:\mathcal{A}\to\RR$, then there exists a unique finite, positive measure $\nu:\mathcal{A}\to[0,\infty)$ and a unique (up to a set of $\nu$ measure zero) $\mathcal{A}$-measurable map $\varrho:X\to\partial\BB_n$ such that $\mu=\varrho\,d\nu$.
\end{LEM}

\begin{proof}[Proof of Lemma \ref{lem_rozklad_miary_varrho_d_nu_ogolniejszy}]
Define a finite, positive measure $\widetilde{\nu}$ as $$\widetilde{\nu}:=|\mu_1|+\ldots+|\mu_n|.$$
Since each $\mu_j$ is absolutely continuous with respect to $\widetilde{\nu}$, from the classical Radon-Nikodym theorem it follows that there exists an $\mathcal{A}$-measurable map $F=(F_1,\ldots,F_n):X\to\RR^n$ such that $F_1,\ldots,F_n\in L^1(X,\widetilde{\nu})$ and $$\mu_j=F_j\,d\widetilde{\nu},\;j=1,\ldots,n.$$
We have $|\mu_j|=|F_j|\,d\widetilde{\nu}$, so
\begin{equation*}
|F_1(x)|+\ldots+|F_n(x)|=1\text{ for }\widetilde{\nu}\text{-a.e. }x\in X.
\end{equation*}
Let $\varrho:X\to\partial\BB_n$ be an $\mathcal{A}$-measurable map such that $F(x)=\varrho(x)\|F(x)\|$ for $\widetilde{\nu}$-a.e. $x\in X$.
Set $\nu:=\|F(x)\|\,d\widetilde{\nu}(x)$.
We have $$\mu=F\,d\widetilde{\nu}=\varrho\,d\nu,$$ what gives a desired decomposition.

It remains to show uniqueness.
Assume that there are $\nu'$, $\varrho'$ satisfying the same conditions as $\nu$, $\varrho$.
We have $\varrho\,d\nu=\varrho'\,d\nu'$.
Set $\omega:=\nu+\nu'$ and let $G,G':X\to[0,\infty)$ be $\mathcal{A}$-measurable functions, integrable with respect to $\omega$ and such that $\nu=G\,d\omega$ and $\nu'=G'\,d\omega$.
We have $$G\varrho\,d\omega=\varrho\,d\nu=\varrho'\,d\nu'=G'\varrho'\,d\omega.$$
Thus, the maps $G\varrho$ and $G'\varrho'$ are equal $\omega$-a.e. on $X$.
This gives $$G(x)=\|G(x)\varrho(x)\|=\|G'(x)\varrho'(x)\|=G'(x)\text{ for }\omega\text{-a.e. }x\in X.$$
In consequence, $\nu=\nu'$ and $\nu$-almost everywhere on $X$ there holds the equality $\varrho=\varrho'$, because $\varrho\,d\nu=\varrho'\,d\nu'$.
\end{proof}

\begin{EX}\label{ex_rozklad_miary_dlebt_alfa_j_delta_j}
In this example we are going to decompose as in Lemma \ref{lem_rozklad_miary_varrho_d_nu} the following $n$-tuple of measures: $$\mu=g\,d\LEBT+(\alpha_1\delta_{\lambda_1},\ldots,\alpha_n\delta_{\lambda_n}),$$ where $g=(g_1,\ldots,g_n)$, $g_1,\ldots,g_n\in L^1(\TT,\LEBT)$, $\alpha_1,\ldots,\alpha_n\in\RR$ and $\lambda_1,\ldots,\lambda_n\in\TT$.
As the measure $\nu$ is required to be singular with respect to $\LEBT$, the first part of desired decomposition is equal to $g\,d\LEBT$ and the second part comes from Lemma \ref{lem_rozklad_miary_varrho_d_nu_ogolniejszy} applied to the measure $(\alpha_1\delta_{\lambda_1},\ldots,\alpha_n\delta_{\lambda_n})$.
To find the latter part, we follow the proof of Lemma \ref{lem_rozklad_miary_varrho_d_nu_ogolniejszy} with $X=\TT$ and $\mathcal{A}$ being the $\sigma$-field of Borel subsets of $\TT$.

For $j\in\lbrace 1,\ldots,n\rbrace$ let $$A_j:=\lbrace l\in\lbrace 1,\ldots,n\rbrace:\lambda_l=\lambda_j\rbrace.$$
We have
$$\widetilde{\nu}=|\alpha_1|\delta_{\lambda_1}+\ldots+|\alpha_n|\delta_{\lambda_n},$$
so we may set
$$F:=\left(\frac{\alpha_1}{\sum_{l\in A_1}|\alpha_l|}\chi_{\lbrace\lambda_1\rbrace},\ldots,\frac{\alpha_n}{\sum_{l\in A_n}|\alpha_l|}\chi_{\lbrace\lambda_n\rbrace}\right)$$
(the mapping $F$ is $\widetilde{\nu}$-almost everywhere well defined, because if $\sum_{l\in A_j}|\alpha_l|=0$ for some $j$, then $\chi_{\lbrace\lambda_j\rbrace}$ is $\widetilde{\nu}$-a.e. equal to $0$ and so is the $j$-th component of the right hand side of the above definition).
Since $\nu=\|F(\lambda)\|\,d\widetilde{\nu}(\lambda)$, the measure $\nu$ is supported on the set $\lbrace\lambda_1,\ldots,\lambda_n\rbrace$ and $$\chi_{\lbrace\lambda_j\rbrace} d\nu=\sqrt{\sum_{l\in A_j}\alpha_l^2}\;\delta_{\lambda_j}.$$
This gives
\begin{equation}\label{eq_exrmlad_miara_nu}
\nu = \sum_{j=1}^{n} \frac{\sqrt{\sum_{l\in A_j}\alpha_l^2}}{\# A_j}\,\delta_{\lambda_j},
\end{equation}
where $\# A_j$ denotes the number of elements of the set $A_j$.

A map $\varrho:\TT\to\partial\BB_n$ has to be taken such that the equality $F(\lambda)=\varrho(\lambda)\|F(\lambda)\|$ holds for $\widetilde{\nu}$-a.e. $\lambda\in\TT$, or equivalently, for $\nu$-a.e. $\lambda\in\TT$.
It means that
\begin{equation}\label{eq_exrmlad_odwzorowanie_ro}
\varrho=\left(\frac{\alpha_1}{\sqrt{\sum_{l\in A_1}\alpha_l^2}}\chi_{\lbrace\lambda_1\rbrace},\ldots,\frac{\alpha_n}{\sqrt{\sum_{l\in A_n}\alpha_l^2}}\chi_{\lbrace\lambda_n\rbrace}\right)\quad\nu\text{-a.e. on }\TT.
\end{equation}
Note that the right hand side is $\nu$-almost everywhere well defined and it does not matter what values $\varrho$ takes outside the set $\lbrace\lambda_1,\ldots,\lambda_n\rbrace$.
The desired decomposition consists of the map $g$, the measure $\nu$ given by \eqref{eq_exrmlad_miara_nu} and a map $\varrho$ satisfying \eqref{eq_exrmlad_odwzorowanie_ro}.

The situation becomes simpler in the case when $\lambda_1,\ldots,\lambda_n$ are pairwise disjoint.
We then have
$$\nu=|\alpha_1|\delta_{\lambda_1}+\ldots+|\alpha_n|\delta_{\lambda_n}$$
and
$$\varrho=\left(\frac{\alpha_1}{|\alpha_1|}\chi_{\lbrace\lambda_1\rbrace},\ldots,\frac{\alpha_n}{|\alpha_n|}\chi_{\lbrace\lambda_n\rbrace}\right)$$
$\widetilde{\nu}$-almost everywhere on $\TT$ (again, $j$-th component of $\varrho$ may be anyhow if $\alpha_j=0$).
\end{EX}

For a convex tube domain $D\subset\CC^n$ introduce the following sets, describing some geometric properties of its base.
Define
\begin{eqnarray*}
W_D &:=& \left\lbrace v\in\RR^n: \sup_{x\in\RE D}\langle x,v\rangle<\infty\right\rbrace,\\
S_D &:=& \left\lbrace y\in\RR^n: \forall v\in W_D:\langle y,v\rangle\leq 0\right\rbrace
\end{eqnarray*}
and for a vector $v\in\RR^n$,
\begin{equation*}
P_D(v):=\lbrace p\in\overline{\RE D}: \langle x-p,v\rangle<0\text{ for all }x\in\RE D\rbrace.
\end{equation*}
It is clear that all these sets are convex, $P_D(v)\subset\partial\RE D$ and if $v\in S_D$, $w\in W_D$ and $t\geq 0$, then $tv\in S_D$ and $tw\in W_D$, i.e. the sets $S_D$ and $W_D$ are infinite cones.
Next observation presents a number of their elementary geometric properties.

\begin{OBS}\label{obs_g_wlasnosci_pd}
Let $D\subset\CC^n$ be a convex tube domain and let $v\in\RR^n$.
Then:
\begin{enumerate}[(i)]
% \item\label{obs_g_wlasnosci_pd_wypukle} the sets $P_D(v)$, $W_D$ and $S_D$ are convex,
% \item\label{obs_g_wlasnosci_pd_stozki} the sets $W_D$ and $S_D$ are infinite cones, i.e. $t S_D\subset S_D$ and $t W_D\subset W_D$ for every $t\geq 0$,
\item\label{obs_g_wlasnosci_pd_pd_domkniety_wypukly} the sets $P_D(v)$ and $S_D$ are closed,
\item\label{obs_g_wlasnosci_pd_niepusty_impl_v_w_wd} if $P_D(v)\neq\varnothing$, then $v\in W_D$,
\item\label{obs_g_wlasnosci_pd_odcinki_w_pd} if $p,q\in P_D(v)$, then the vectors $p-q$ and $v$ are orthogonal,
% \item\label{obs_g_wlasnosci_pd_wd_wypukly} the set $W_D$ is convex,
% \item\label{obs_g_wlasnosci_pd_granica_pm_w_pd_v} if $(v_m)_{m=1}^\infty\in\RR^n$, $v_m\to v$, $v\neq 0$, $p_m\in P_D(v_m)$ and $p_m\to p\in\overline{\RE D}$, then $p\in P_D(v)$.
% \item\label{obs_g_wlasnosci_pd_6} if $D\in\mathcal{D}_n$, then $W_D\subset[0,\infty)^n$ and $e_1,\ldots,e_n\in W_D$,
\item\label{obs_g_wlasnosci_pd_obszar_scisle_wypukly} if the domain $\RE D$ is strictly convex (in the geometric sense, i.e. it is convex and $\partial\RE D$ does not contain any non-trivial segments), then the set $P_D(v)$ contains at most one element,
% \item\label{obs_g_wlasnosci_pd_sd_stozek_wypukly} the set $S_D$ is a closed, convex, infinite cone with vertex at the origin, privided that $S\neq\lbrace 0\rbrace$,
% \item\label{obs_g_wlasnosci_pd_v_w_brzegu_wd_iff_sd_cap_v_prostop_nietryw} if $v\in\overline{W_D}$, then $S_D\cap\lbrace v\rbrace^{\perp}\neq\lbrace 0\rbrace$ iff $v\in\partial W_D$,
\item\label{obs_g_wlasnosci_pd_elementy_sd_jako_kierunki} $v\in S_D$ iff for all $a\in\RE D$ and $t\geq 0$ there holds $a+tv\in \RE D$,
\item\label{obs_g_wlasnosci_pd_wd_niepuste_wnetrze} if $\RE D$ contains no complex affine lines, then $\INT W_D\neq\varnothing$,
\item\label{obs_g_wlasnosci_pd_baza_ograniczona} if $\RE D$ is bounded, then $W_D=\RR^n$ and $S_D=\lbrace 0\rbrace$.
\end{enumerate}
\end{OBS}

\begin{proof}
\eqref{obs_g_wlasnosci_pd_pd_domkniety_wypukly}. If $(p_m)_m\subset P_D(v)$ and $p_m\to p$, then $\langle x-p,v\rangle\leq 0$ for each $x\in\RE D$.
As $P_D(v)\neq\varnothing$, we have $v\neq 0$, so the map $x\mapsto\langle x-p,v\rangle$ is open.
It is non-positive on the open set $\RE D$, so it is in fact negative on $\RE D$.

\eqref{obs_g_wlasnosci_pd_odcinki_w_pd}. If $p,q\in P_D(v)$, then $\frac12(p+q)\in P_D(v)$.
Since $p,q\in\overline{\RE D}$, we have $\langle p-\frac12(p+q),v\rangle\leq 0$ and $\langle q-\frac12(p+q),v\rangle\leq 0$, what gives $\langle p-q,v\rangle=0$.

\eqref{obs_g_wlasnosci_pd_wd_niepuste_wnetrze}. It follows e.g. from \cite[Observation 2.4]{zajac}.

\eqref{obs_g_wlasnosci_pd_niepusty_impl_v_w_wd}, \eqref{obs_g_wlasnosci_pd_obszar_scisle_wypukly}, \eqref{obs_g_wlasnosci_pd_elementy_sd_jako_kierunki}, \eqref{obs_g_wlasnosci_pd_baza_ograniczona}.
The proofs are immediate.
\end{proof}

\section{Description of complex geodesics in an arbitrary convex tube domain and its applications in special classes of domains}\label{sect_special_classes}

In this section we formulate the main result of this paper, Theorem \ref{th_postac_geod}.
It gives a full description of all complex geodesics for $D$ in terms of its geometric properties, i.e. the sets $P_D(v)$, $W_D$, $S_D$.
In the latter part of this section we show how it can be applied to obtain formulas for boundary measures of complex geodesics in some special classes of convex tube domains.
We also give some examples.
The proof of Theorem \ref{th_postac_geod} is presented in Section \ref{sect_dowod_tw_postac_geod}.

\begin{THE}\label{th_postac_geod}
Let $D\subset\CC^n$ be a convex tube domain containing no complex affine lines and let $\varphi\in\MMM^n$ be a holomorphic map with boundary measure $\mu$.
Consider the decomposition
\begin{equation*}
\mu = g\,d\LEBT+\varrho\,d\nu,
\end{equation*}
where $g=(g_1,\ldots,g_n):\TT\to\RR^n$ and $\varrho:\TT\to\partial\BB_n$ are Borel-measurable maps, $g_1,\ldots,g_n\in L^1(\TT,\LEBT)$ and $\nu$ is a positive, finite, Borel measure on $\TT$ singular to $\LEBT$.

Then
\begin{center}
$\varphi(\DD)\subset D$ and $\varphi$ is a complex geodesic for $D$
\end{center}
iff there exists a map $h\in\HH^n$, $h\not\equiv 0$, such that the following conditions hold:
\begin{enumerate}[(i)]
\item\label{th_postac_geod_g_in_pd_h} $g(\lambda)\in P_D(\BLHL{})$ for $\LEBT$-a.e. $\lambda\in\TT$,
\item\label{th_postac_geod_h_bullet_rho_wieksze} $\langle\BLHL{},\varrho(\lambda)\rangle\geq 0$ for $\nu$-a.e. $\lambda\in\TT$,
\item\label{th_postac_geod_rho_w_SD} $\varrho(\lambda)\in S_D$ for $\nu$-a.e. $\lambda\in\TT$,
\item\label{th_postac_geod_fi_0} $\RE\varphi(0)\in\RE D$.
\end{enumerate}

Moreover, if $\varphi(\DD)\subset D$, $\varphi$ is a complex geodesic for $D$ and $h\in\HH^n$, $h\not\equiv 0$ is a map satisfying the conditions \eqref{th_postac_geod_g_in_pd_h} - \eqref{th_postac_geod_fi_0}, then there also hold:
\begin{enumerate}[(i)]
\setcounter{enumi}{4}
\item\label{th_postac_geod_h_bullet_rho_rowne_0} $\varrho(\lambda)\in S_D\cap\lbrace \BLHL{}\rbrace^{\perp}$ for $\nu$-a.e. $\lambda\in\TT$,
\item\label{th_postac_geod_nu_skupione_na} $\nu(\lbrace\lambda\in\TT:\BLHL{}\in\INT W_D\rbrace)=0$.
\item\label{th_postac_geod_blhl_in_WD} $\BLHL{}\in\overline{W_D}$ for every $\lambda\in\TT$.
\end{enumerate}
\end{THE}

\noindent
Note that from \eqref{th_postac_geod_nu_skupione_na} and \eqref{th_postac_geod_blhl_in_WD} it follows that the measure $\nu$ is supported on the set $\lbrace\lambda\in\TT:\BLHL{}\in\partial W_D\rbrace$.

\begin{RM}\label{rm_separate_conditions_and_construction}
Theorem \ref{th_postac_geod} gives quite separate conditions for both parts $g\,d\LEBT$ and $\varrho\,d\nu$ of the decomposition of $\mu$, what makes it relatively not difficult to construct a measure which defines a complex geodesic for $D$.
The part $g\,d\LEBT$ must satisfy \eqref{th_postac_geod_g_in_pd_h}, while the part $\varrho\,d\nu$ must fulfill \eqref{th_postac_geod_h_bullet_rho_wieksze} and \eqref{th_postac_geod_rho_w_SD}.
Everything is connected 'only' by the map $h$.
To construct a measure $\mu$ which defines a complex geodesic for $D$ it suffices to choose a map $h\in\HH^n$, $h\not\equiv 0$ such that
\begin{equation}\label{eq_pd_blhl_niepusty}
P_D(\BLHL{})\neq\varnothing\text{ for }\LEBT\text{-a.e. }\lambda\in\TT 
\end{equation}
and next:
\begin{itemize}
\item take a Borel map $g$ with integrable components satisfying \eqref{th_postac_geod_g_in_pd_h} (note that it may happen that it is impossible, even if \eqref{eq_pd_blhl_niepusty} holds - see Example \ref{ex_obszar_1_nad_x}),
\item take a measure $\nu$ singular to $\LEBT$ and satisfying \eqref{th_postac_geod_nu_skupione_na},
\item take a Borel map $\varrho:\TT\to\partial\BB_n$ satisfying \eqref{th_postac_geod_h_bullet_rho_rowne_0}.
\end{itemize}
Then, if $\mu=g\,d\LEBT+\varrho\,d\nu$ and additionally $\frac{1}{2\pi}\mu(\TT)\in\RE D$ (i.e. $\RE\varphi(0)\in\RE D$), then $\mu$ is a boundary measure of a complex geodesic for the domain $D$.
\end{RM}

Before we proceed to applications of Theorem \ref{th_postac_geod}, let us make a remark on convex tubes with bounded base:

\begin{RM}\label{rm_th_postac_geod_if_re_D_ograniczone}
If $D\subset\CC^n$ is a convex tube domain containing no complex affine lines and $\RE D$ is bounded, then $W_D=\RR^n$ and $S_D=\lbrace 0\rbrace$, so from the condition \eqref{th_postac_geod_rho_w_SD} of Theorem \ref{th_postac_geod} it follows that $\varrho(\lambda)=0$ for $\nu$-a.e. $\lambda\in\TT$.
Hence $\nu$ is a null measure, because the image of $\varrho$ lies in $\partial\BB_n$.
Then also the condition \eqref{th_postac_geod_h_bullet_rho_wieksze} is automatically fulfilled.
Thus, a holomorphic map $\varphi$ with boundary measure $\mu$ is a complex geodesic for $D$ iff $$\mu=g\,d\LEBT$$ for some $g$, $h$ satisfying \eqref{th_postac_geod_g_in_pd_h} and \eqref{th_postac_geod_fi_0}.

In general case, i.e. for $D$ with unbounded base, the absolutely continuous part of $\mu$ is determined by Theorem \ref{th_postac_geod} \eqref{th_postac_geod_g_in_pd_h} in the same way.
But in that case we have to deal also with the singular part of $\mu$, what is a more interesting issue.
\end{RM}

\subsection{Convex tube domains with $\overline{W_D}=[0,\infty)^n$}\label{subsect_domains_Dn}

In this section we investigate the family $\mathcal{D}_n$ of all convex tube domains $D\subset\CC^n$ such that $$\overline{W_D}=[0,\infty)^n.$$
A convex tube domain $D$ belongs to $\mathcal{D}_n$ iff $e_1,\ldots,e_n\in\overline{W_D}$ and $$\RE D+(-\infty,0]^n=\RE D.$$
The base of such a domain $D$ contains no real affine lines and there holds the equality $$S_D=(-\infty,0]^n.$$
In Corollary \ref{crl_postac_geod_Dn} we describe all complex geodesics for a domain $D\in\mathcal{D}_n$ and we apply it in Section \ref{sect_reinhardt_domains} to describe extremal mappings in some classes of Reinhardt domains in $\CC^n$.

\begin{CRL}\label{crl_postac_geod_Dn}
Let $D\in\mathcal{D}_n$, $n\geq 2$, and let $\varphi\in\MMM^n$ be a holomorphic map with boundary measure $\mu$.
Consider the decomposition
\begin{equation*}
\mu = g\,d\LEBT+\varrho\,d\nu,
\end{equation*}
where $g=(g_1,\ldots,g_n):\TT\to\RR^n$ and $\varrho=(\varrho_1,\ldots,\varrho_n):\TT\to\partial\BB_n$ are Borel-measurable maps, $g_1,\ldots,g_n\in L^1(\TT,\LEBT)$ and $\nu$ is a positive, finite, Borel measure on $\TT$ singular to $\LEBT$.
Then
\begin{center}
$\varphi(\DD)\subset D$ and $\varphi$ is a complex geodesic for $D$
\end{center}
iff there exists a map $h\in\HH^n$, $h\not\equiv 0$ such that the following conditions hold:
\begin{enumerate}[(i)]
\item\label{crl_postac_geod_Dn_hj_w_Hplus} $h\in\HHP^n$,
\item\label{crl_postac_geod_Dn_g_in_pd_h} $g(\lambda)\in P_D(\BLHL{})$ for $\LEBT$-a.e. $\lambda\in\TT$.
\item\label{crl_postac_geod_Dn_rho_in_minus_infty_0} $\varrho(\lambda)\in(-\infty,0]^n$ for $\nu$-a.e. $\lambda\in\TT$,
\item\label{crl_postac_geod_Dn_fi_0} $\RE\varphi(0)\in\RE D$,
\item\label{crl_postac_geod_Dn_varrho_j_nu_if_hj_nie_zero} if $j\in\lbrace 1,\ldots,n\rbrace$ is such that $h_j\not\equiv 0$, then $$\varrho_j\,d\nu=\alpha_j\delta_{\lambda_j}$$ for some $\lambda_j\in\TT$ and $\alpha_j\in(-\infty,0]$ such that $\alpha_j h_j(\lambda_j)=0$.
\end{enumerate}
\end{CRL}

\noindent
Note that the condition \eqref{crl_postac_geod_Dn_rho_in_minus_infty_0} from the above Corollary means that the singular part of $\mu$, i.e. the measure $\varrho\,d\nu$, is just a $n$-tuple of negative measures.
Moreover, the condition \eqref{crl_postac_geod_Dn_varrho_j_nu_if_hj_nie_zero} means that if $h_j\not\equiv 0$, then the $j$-th component of the singular part of $\mu$ is of the form $\alpha_j\delta_{\lambda_j}$ with some appropriate $\alpha_j$ and $\lambda_j$.

\begin{proof}[Proof of Corollary \ref{crl_postac_geod_Dn}]
Assume that $\varphi(\DD)\subset D$ and $\varphi$ is a complex geodesic for $D$.
Let $h$ be as in Theorem \ref{th_postac_geod}.
The conditions \eqref{crl_postac_geod_Dn_hj_w_Hplus} - \eqref{crl_postac_geod_Dn_fi_0} follow directly from Theorem \ref{th_postac_geod}, so it remains to show the condition \eqref{crl_postac_geod_Dn_varrho_j_nu_if_hj_nie_zero}.

Set $\varrho=(\varrho_1,\ldots,\varrho_n)$.
The expression $\langle\BLHL{},\varrho(\lambda)\rangle$, which is by Theorem \ref{th_postac_geod} \eqref{th_postac_geod_h_bullet_rho_rowne_0} $\nu$-almost everywhere equal to zero, is a sum of $\nu$-almost everywhere non-positive terms $\BLHL{1}\varrho_1(\lambda),\ldots,\BLHL{n}\varrho_n(\lambda)$.
Therefore, all this terms are $\nu$-a.e. equal to zero.
If $j$ is such that $h_j\not\equiv 0$, then the function $h_j\in\HHP^1$ has at most one root on $\TT$ (counting without multiplicities).
Hence, up to a set of $\nu$ measure zero, $\varrho_j=\beta_j\chi_{\lbrace\lambda_j\rbrace}$ for some $\lambda_j\in\TT$ and $\beta_j\in(-\infty,0]$ such that $\beta_j h_j(\lambda_j)=0$.
This gives the condition \eqref{crl_postac_geod_Dn_varrho_j_nu_if_hj_nie_zero} with $\alpha_j:=\beta_j\nu(\lbrace\lambda_j\rbrace)$.

On the other hand, it follows from Theorem \ref{th_postac_geod} that if $h$ is such that the conditions \eqref{crl_postac_geod_Dn_hj_w_Hplus} - \eqref{crl_postac_geod_Dn_varrho_j_nu_if_hj_nie_zero} are satisfied, then $\varphi$ is a complex geodesic for $D$.
Indeed, is is clear that the conditions \eqref{th_postac_geod_g_in_pd_h}, \eqref{th_postac_geod_rho_w_SD} and \eqref{th_postac_geod_fi_0} from Theorem \ref{th_postac_geod} hold, so it suffices to show that \eqref{th_postac_geod_h_bullet_rho_wieksze} is also fulfilled.
From the assumption \eqref{crl_postac_geod_Dn_varrho_j_nu_if_hj_nie_zero} we conclude that if $j$ is such that $h_j\not\equiv 0$, then $$\BLHL{j}\varrho_j(\lambda)\,d\nu(\lambda)=\alpha_j\BLHL{j}\,d\delta_{\lambda_j}(\lambda)=\alpha_j\bar\lambda_j h_j(\lambda_j)\,d\delta_{\lambda_j}(\lambda)=0.$$
This implies that $\langle\BLHL{},\varrho(\lambda)\rangle\,d\nu(\lambda)$ is a null measure, what involves the condition \eqref{th_postac_geod_h_bullet_rho_wieksze}.
The proof is complete.
\end{proof}

\begin{RM}\label{rm_do_crl_postac_geod_Dn}
Under the assumptions of Corollary \ref{crl_postac_geod_Dn}, if $\varphi$ is a complex geodesic for $D$, $h$ is as in the corollary and $h_1\not\equiv 0,\ldots,h_n\not\equiv 0$, then it follows from \eqref{crl_postac_geod_Dn_varrho_j_nu_if_hj_nie_zero} that $$\varrho\,d\nu=(\alpha_1\delta_{\lambda_1},\ldots,\alpha_n\delta_{\lambda_n})$$ for some $\alpha_1,\ldots,\alpha_n\in(-\infty,0]$ and $\lambda_1,\ldots,\lambda_n\in\TT$ such that $$\alpha_1 h_1(\lambda_1)=\ldots=\alpha_n h_n(\lambda_n)=0.$$
Thus, the singular part of $\mu$ takes then a very special form.

In the opposite situation, i.e. when the set $A:=\lbrace j\in\lbrace 1,\ldots,n\rbrace:h_j\equiv 0\rbrace$ is non-empty, for every $j\in A$ the $j$-th component of the singular part of $\mu$ may be almost arbitrary.
More precisely, if $\omega_1,\ldots,\omega_n$ are finite, negative Borel measures on $\TT$, singular to $\LEBT$ and such that $\omega_j=\varrho_j\,d\nu$ for every $j\in\lbrace 1,\ldots,n\rbrace\setminus A$, then a holomorphic map $\psi$ with boundary measure $g\,d\LEBT+(\omega_1,\ldots,\omega_n)$ is a complex geodesic for $D$, provided that $\RE\psi(0)\in\RE D$.
This fact follows directly from Corollary \ref{crl_postac_geod_Dn}, because $\psi$ satisfies the conditions \eqref{crl_postac_geod_Dn_hj_w_Hplus} - \eqref{crl_postac_geod_Dn_varrho_j_nu_if_hj_nie_zero} with $h$.

Note that if the domain $\RE D$ is strictly convex (in the geometric sense), then there must hold $h_1\not\equiv 0,\ldots,h_n\not\equiv 0$.
Indeed, in view of the condition \eqref{crl_postac_geod_Dn_g_in_pd_h} from Corollary \ref{crl_postac_geod_Dn}, $\LEBT$-almost all sets $P_D(\BLHL{})$ are non-empty.
Our claim is a consequence of strict convexity of $D$ and of the following geometric property of domains from the family $\mathcal{D}_n$: if for a vector $v=(v_1,\ldots,v_n)\in\RR^n$ we have $P_D(v)\neq\varnothing$ and $v_j=0$ for some $j$, then $\partial\RE D$ contains a half-line of the form $p+(-\infty,0]\,e_j$ for any $p\in P_D(v)$.
\end{RM}

\begin{EX}\label{ex_obszar_z_rodziny_D2_z_cwiartka_kola}
Consider the following domain from the family $\mathcal{D}_2$, $$D:=\lbrace (x_1,x_2)\in\RR^2: (\max\lbrace x_1+1,0\rbrace)^2+(\max\lbrace x_2+1,0\rbrace)^2<1\rbrace+i\RR^2.$$
The base of $D$ is shown on Figure \ref{fig_obszar_z_D2_z_cwiartka_kola}.
\begin{figure}[htb]
\centering
\includegraphics[scale=0.4]{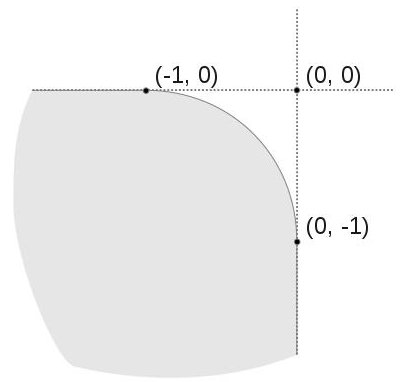}
\caption{\protect\label{fig_obszar_z_D2_z_cwiartka_kola}The base of $D$}
\end{figure}
We have
\begin{equation*}
P_D(v) = \left\lbrace\begin{array}{ll}
\left\lbrace \frac{v}{\|v\|}-(1,1)\right\rbrace, &\text{if } v\in(0,\infty)^2,\\
(-\infty,-1]\times\lbrace 0\rbrace, &\text{if } v\in\lbrace 0\rbrace\times(0,\infty),\\
\lbrace 0\rbrace\times(-\infty,-1], &\text{if } v\in(0,\infty)\times\lbrace 0\rbrace,\\
\varnothing, &\text{otherwise}.
\end{array}\right.
\end{equation*}

Take a complex geodesic $\varphi=(\varphi_1,\varphi_2):\DD\to D$ with boundary measure $\mu=(\mu_1,\mu_2)$ and let $g=(g_1,g_2)$, $\nu$, $\varrho=(\varrho_1,\varrho_2)$ and $h=(h_1,h_2)\in\HHP^2$ be as in Corollary \ref{crl_postac_geod_Dn}.

Assume that the functions $h_1$, $h_2$ are linearly independent.
By Corollary \ref{crl_postac_geod_Dn} \eqref{crl_postac_geod_Dn_g_in_pd_h}, for $\LEBT$-a.e. $\lambda\in\TT$ we have $$g(\lambda)\in P_D(\BLHL{})=\left\lbrace\frac{\BLHL{}}{\|\BLHL{}\|}-(1,1)\right\rbrace.$$
In view of Remark \ref{rm_do_crl_postac_geod_Dn}, we obtain
\begin{equation}\label{eq_eod2ck_mu_if_h1_h2_not_equiv_0}
\mu = \left(\frac{\BLHL{}}{\|\BLHL{}\|}-(1,1)\right)d\LEBT(\lambda)+(\alpha_1\delta_{\lambda_1},\alpha_2\delta_{\lambda_2}) 
\end{equation}
for some $\alpha_1,\alpha_2\in(-\infty,0]$ and $\lambda_1,\lambda_2\in\TT$ such that
\begin{equation}\label{eq_eod2ck_alpha_j_hj_delta_j_h1_h2_not_equiv_0}
\alpha_1 h_1(\lambda_1)=\alpha_2 h_2(\lambda_2)=0. 
\end{equation}
On the other hand, from Corollary \ref{crl_postac_geod_Dn} it follows that if a holomorphic map has boundary measure of the form \eqref{eq_eod2ck_mu_if_h1_h2_not_equiv_0} with some linearly independent $h_1,h_2\in\HHP^2$ and some $\alpha_1,\alpha_2\leq 0$, $\lambda_1,\lambda_2\in\TT$ satisfying \eqref{eq_eod2ck_alpha_j_hj_delta_j_h1_h2_not_equiv_0}, then it is a complex geodesic for $D$ (the condition \eqref{crl_postac_geod_Dn_fi_0} from Corollary \ref{crl_postac_geod_Dn} is then a consequence of linear independence of $h_1$, $h_2$).

Now consider the case when $h_1,h_2$ are linearly dependent, but $h_1,h_2\not\equiv 0$.
We have $h_2=\gamma h_1$ for some $\gamma>0$.
Since $P_D(\BLHL{})=P_D((1,\gamma))$ for $\LEBT$-a.e. $\lambda\in\TT$, the map $g$ is $\LEBT$-almost everywhere constant and equal to $\frac{(1,\gamma)}{\|(1,\gamma)\|}-(1,1)$.
Using Remark \ref{rm_do_crl_postac_geod_Dn} again we conclude that the singular part of $\mu$ is of the form $(\alpha_1\delta_{\lambda_1},\alpha_2\delta_{\lambda_2})$
for some $\alpha_1,\alpha_2\in(-\infty,0]$ and $\lambda_1,\lambda_2\in\TT$ such that $\alpha_1 h_1(\lambda_1)=\alpha_2\gamma h_1(\lambda_2)=0$.
But from Corollary \ref{crl_postac_geod_Dn} \eqref{crl_postac_geod_Dn_fi_0} it follows that $(\alpha_1,\alpha_2)\neq (0,0)$, so $h_1$, $h_2$ have a common root $\lambda_0\in\lbrace\lambda_1,\lambda_2\rbrace$.
Thus, 
\begin{equation}\label{eq_eod2ck_mu_if_h1_h2_dependent_but_not_equiv_0}
\mu = \left(\frac{(1,\gamma)}{\|(1,\gamma)\|}-(1,1)\right)d\LEBT+(\alpha_1,\alpha_2)\delta_{\lambda_0}.
\end{equation}
On the other hand, if a holomorphic map has boundary measure of the form \eqref{eq_eod2ck_mu_if_h1_h2_dependent_but_not_equiv_0} with some $\gamma>0$, $\lambda_0\in\TT$ and $\alpha_1,\alpha_2\leq 0$ such that $\alpha_1+\alpha_2<0$, then it is a complex geodesic for $D$.
It is a consequence of Corollary \ref{crl_postac_geod_Dn} applied to that map and to $h(\lambda)=(\lambda-\lambda_0)(1-\bar\lambda_0\lambda)\cdot(1,\gamma)$.

It remains to consider the situation when $h_1\equiv 0$ or $h_2\equiv 0$.
If $h_1\equiv 0$, then $g(\lambda)\in(-\infty,-1]\times\lbrace 0\rbrace$ for $\LEBT$-a.e. $\lambda\in\TT$, in view of \eqref{crl_postac_geod_Dn_g_in_pd_h}.
Moreover, from \eqref{crl_postac_geod_Dn_varrho_j_nu_if_hj_nie_zero} it follows that $\varrho_2\,d\nu=\alpha_2\delta_{\lambda_2}$ for some $\alpha_2\leq 0$ and $\lambda_2\in\TT$ such that $\alpha_2 h_2(\lambda_2)=0$, and \eqref{crl_postac_geod_Dn_fi_0} gives that $\alpha_2\neq 0$.
Therefore
\begin{equation}\label{eq_eod2ck_mu_if_h1_equiv_0}
\mu = (\mu_1,\alpha_2\delta_{\lambda_2}).
\end{equation}
On the other hand, one can check that if a holomorphic map has boundary measure given by \eqref{eq_eod2ck_mu_if_h1_equiv_0} with some $\alpha_2<0$, $\lambda_2\in\TT$ and some real Borel measure $\mu_1$ on $\TT$ such that $\mu_1\leq-\LEBT$, then it is a complex geodesic for the domain $D$.
If $h_2\equiv 0$, then arguing similarly as before we obtain that
\begin{equation}\label{eq_eod2ck_mu_if_h2_equiv_0}
\mu = (\alpha_1\delta_{\lambda_1},\mu_2)
\end{equation}
for some $\alpha_1<0$, $\lambda_1\in\TT$ and $\mu_2\leq-\LEBT$.
Any holomorphic map with boundary measure of this form is a complex geodesic for $D$.
\end{EX}

\begin{EX}\label{ex_obszar_1_nad_x}
Consider the domain $$D:=\lbrace (x_1,x_2)\in\RR^2:x_1,x_2<0, x_1 x_2>1\rbrace+i\RR^2.$$
It belongs to the family $\mathcal{D}_2$.
We have
\begin{equation}\label{eq_eo1x_PD}
P_D(v)=\left\lbrace\left(-\sqrt{\frac{v_2}{v_1}},-\sqrt{\frac{v_1}{v_2}}\right)\right\rbrace
\end{equation}
when $v=(v_1,v_2)\in(0,\infty)^2$ and $P_D(v)=\varnothing$ otherwise.

Take a complex geodesic $\varphi:\DD\to D$ with boundary measure $\mu$ and let $g$, $\nu$, $\varrho$ and $h=(h_1,h_2)\in\HHP^2$ be as in Corollary \ref{crl_postac_geod_Dn}.
Assume that $h_1$ and $h_2$ are linearly independent.
Then $h_j(\lambda)=c_j(\lambda-\lambda_j)(1-\bar \lambda_j \lambda)$ for some $c_1,c_2>0$, $\lambda_1,\lambda_2\in\CDD$ such that $\lambda_1\neq \lambda_2$.
By Corollary \ref{crl_postac_geod_Dn} \eqref{crl_postac_geod_Dn_g_in_pd_h}, for $\LEBT$-a.e. $\lambda\in\TT$ we have $$g(\lambda)=\left(-\sqrt{\frac{c_2}{c_1}}\,\frac{|\lambda-\lambda_2|}{|\lambda-\lambda_1|},-\sqrt{\frac{c_1}{c_2}}\,\frac{|\lambda-\lambda_1|}{|\lambda-\lambda_2|}\right).$$
Since both components of $g$ belong to $L^1(\TT,\LEBT)$, we have $\lambda_1,\lambda_2\in\DD$.
Moreover, $\varrho\,d\nu=0$, by Remark \ref{rm_do_crl_postac_geod_Dn}.
In summary,
\begin{equation}\label{eq_eo1x_mu_if_h12_niezal}
\mu=\left(c\,\frac{|\lambda-\lambda_2|}{|\lambda-\lambda_1|},\frac1c\,\frac{|\lambda-\lambda_1|}{|\lambda-\lambda_2|}\right)d\LEBT(\lambda),
\end{equation}
where $c=-(\frac{c_2}{c_1})^{\frac12}$.
Such a map $\varphi$ extends analytically on a neighbourhood of $\CDD$, because the map $g$ is real analytic on $\TT$.
It follows from Corollary \ref{crl_postac_geod_Dn} that any holomorphic map with boundary measure of the form \eqref{eq_eo1x_mu_if_h12_niezal} (with some $c>0$ and $\lambda_1,\lambda_2\in\DD$, $\lambda_1\neq\lambda_2$) is a complex geodesic for $D$.

If the functions $h_1$, $h_2$ are linearly dependent, then similarly as in Example \ref{ex_obszar_z_rodziny_D2_z_cwiartka_kola} we can show that $\mu$ is of the form
\begin{equation}\label{eq_eo1x_mu_if_h12_zal}
\mu=(-\gamma^{\frac12},-\gamma^{-\frac12})\,d\LEBT+(\alpha_1,\alpha_2)\delta_{\lambda_0}
\end{equation}
for some $\gamma>0$, $\lambda_0\in\TT$ and $\alpha_1,\alpha_2\leq 0$ such that $\alpha_1+\alpha_2<0$.
And again, any holomorphic map with boundary measure of the above form is a complex geodesic for $D$.

We see that in this example every complex geodesic which admits a map $h$ with linearly independent components can be extended analytically on a neighbourhood of the closed unit disc $\CDD$.
However, even in some 'similar' domains this property do not hold.
For example, let $$D':=\left\lbrace (x_1,x_2)\in\RR^2:x_1,x_2<0, x_2<-x_1^{-2}\right\rbrace+i\RR^2.$$
For $v=(v_1,v_2)\in(0,\infty)^2$ we have $$P_{D'}(v)=\left\lbrace\left(-\left(\frac{2v_2}{v_1}\right)^{\frac13}, -\left(\frac{v_1}{2v_2}\right)^{\frac23}\right)\right\rbrace.$$
Take $h(\lambda):=((\lambda+1)^2,\lambda)$ (it belongs to the family $\HHP^2$) and $g$ such that $g(\lambda)\in P_{D'}(\BLHL{})$ for $\LEBT$-a.e. $\lambda\in\TT$, i.e.
$$g(\lambda)=\left(-2^{\frac13}|\lambda+1|^{-\frac23}, -2^{-\frac23}|\lambda+1|^{\frac43}\right).$$
We see that both components of $g$ belong to $L^1(\TT)$.
From Corollary \ref{crl_postac_geod_Dn} it follows that if $\alpha_1\leq 0$, then the holomorphic map given by the boundary measure $\mu:=g\,d\LEBT+(\alpha_1\delta_{-1},0)$ is a complex geodesic for $D'$.
But this map does not extend analytically to a neighbourhood of $\CDD$.

In these examples we also see that it is possible that for some $h$ there is no map $g$ with components integrable with respect to $\LEBT$ and satisfying $g(\lambda)\in P_D(\BLHL{})$ for $\LEBT$-a.e. $\lambda\in\TT$, even if these sets are non-empty for $\LEBT$-a.e. $\lambda\in\TT$ (cf. Remark \ref{rm_separate_conditions_and_construction}).
\end{EX}

\begin{RM}
Although in examples presented above we considered only tube domains with base contained in $(-\infty,0)^n$, it is clear that in the family $\mathcal{D}_n$ there are domains with base not contained in any set of the form $a+(-\infty,0)^n$, $a\in\RR^n$.
An example of such a domain in $\CC^2$ is $\lbrace (x_1,x_2)\in\RR^2: x_2<-e^{x_1}\rbrace+i\RR^2$, where $W_D=\lbrace (0,0)\rbrace\cup [0,\infty)\times(0,\infty)\subsetneq\overline{W_D}$.
Applying Corollary \ref{crl_postac_geod_Dn} in the same way as previously, we can find formulas for boundary measures of all complex geodesics for such tube domains.
\end{RM}

\subsection{Domains in $\CC^2$}\label{subsect_domains_in_c2}

Let $D\subset\CC^2$ be a convex tube domain containing no complex affine lines.
From Observation \ref{obs_g_wlasnosci_pd} it follows that the set $\overline{W_D}$ is a closed, convex, infinite cone with vertex at the origin and with non-empty interior.
Thus, $\overline{W_D}$ is any of the whole $\RR^2$, a half-plane or a convex infinite angle, i.e. the set $$\lbrace (r\cos\theta,r\sin\theta):r\geq 0, \theta\in[\theta_1,\theta_2]\rbrace$$ for some $\theta_1<\theta_2<\theta_1+\pi$.

If $\overline{W_D}$ is the whole $\RR^2$, then $\RE D$ is bounded.
Tubes with bounded base were considered in Remark \ref{rm_th_postac_geod_if_re_D_ograniczone}.
If $\overline{W_D}$ is an angle, then $D$ is affinely equivalent to a convex tube domain $D'\subset\CC^2$ having $\overline{W_{D'}}=[0,\infty)^2$.
These domains are exactly those from the family $\mathcal{D}_2$ and they were considered in Subsection \ref{subsect_domains_Dn}.
If $\overline{W_D}$ is a half-plane, then we may assume that $\overline{W_D}=\RR\times(-\infty,0]$.
This case we consider now, in Corollary \ref{crl_wzor_geod_w_c2_if_WD_polplaszczyzna}.
If $D\subset\CC^2$ is a convex tube domain with $\overline{W_D}=\RR\times(-\infty,0]$, then $D$ contains no complex affine lines, there holds the equality $$S_D=\lbrace 0\rbrace\times[0,\infty)$$ and $D$ is of the form $$D=\lbrace (x_1,x_2)\in(a,b)\times\RR: x_2>f(x_1)\rbrace+i\RR^2$$ for some $-\infty\leq a<b\leq\infty$ and a convex function $f:(a,b)\to\RR$ such that:
\begin{itemize}
\item if $a=-\infty$, then $f'_-(x), f'_+(x)\to-\infty$, when $x\to-\infty$, and
\item if $b=\infty$, then $f'_-(x), f'_+(x)\to\infty$, when $x\to\infty$.
\end{itemize}
Here $f'_-$ and $f'_+$ denotes the one-sided derivatives of $f$.
Depending on $a$, $b$ and $f$, the set $\overline{W_D}\setminus W_D$ may be any of the empty set, a horizontal half-line starting at the origin or the horizontal line $\RR\times\lbrace 0\rbrace$.
In Corollary \ref{crl_wzor_geod_w_c2_if_WD_polplaszczyzna} all of this cases are treated the same, as there only the set $\overline{W_D}$ is important, not $W_D$ itself.

\begin{CRL}\label{crl_wzor_geod_w_c2_if_WD_polplaszczyzna}
Let $D\subset\CC^2$ be a convex tube domain such that $\overline{W_D}=\RR\times(-\infty,0]$.
Take a map $\varphi\in\MMM^2$ with boundary measure $\mu$ and consider the decomposition
\begin{equation*}
\mu = g\,d\LEBT+\varrho\,d\nu,
\end{equation*}
where $g=(g_1,g_2):\TT\to\RR^2$ and $\varrho:\TT\to\partial\BB_2$ are Borel-measurable maps, $g_1,g_2\in L^1(\TT,\LEBT)$ and $\nu$ is a positive, finite, Borel measure on $\TT$ singular to $\LEBT$.
Then
\begin{center}
$\varphi(\DD)\subset D$ and $\varphi$ is a complex geodesic for $D$
\end{center}
iff there exists a map $h\in\HH^2$, $h\not\equiv 0$ such that the following conditions hold:
\begin{enumerate}[(i)]
\item\label{crl_wzor_geod_w_c2_if_WD_polplaszczyzna_h2_w_minus_Hplus} $h_2\in-\HHP^1$,
\item\label{crl_wzor_geod_w_c2_if_WD_polplaszczyzna_g_in_pd_h} $g(\lambda)\in P_D(\BLHL{})$ for $\LEBT$-a.e. $\lambda\in\TT$,
\item\label{crl_wzor_geod_w_c2_if_WD_polplaszczyzna_rho_pionowe} $\varrho(\lambda)=e_2$ for $\nu$-a.e. $\lambda\in\TT$,
\item\label{crl_wzor_geod_w_c2_if_WD_polplaszczyzna_fi_0} $\RE\varphi(0)\in\RE D$,
\item\label{crl_wzor_geod_w_c2_if_WD_polplaszczyzna_varrho_nu_if_h2_nie_zero} if $h_2\not\equiv 0$, then $\nu=\alpha\delta_{\lambda_0}$ for some $\alpha\in [0,\infty)$ and $\lambda_0\in\TT$ such that $\alpha h_2(\lambda_0)=0$.
\end{enumerate}
\end{CRL}

\noindent
The condition \eqref{crl_wzor_geod_w_c2_if_WD_polplaszczyzna_rho_pionowe} from the above Corollary means that $\varrho\,d\nu=(0,\nu)$.
In particular, $\nu$ is equal to the singular part of the second component of $\mu$.

\begin{proof}[Proof of Corollary \ref{crl_wzor_geod_w_c2_if_WD_polplaszczyzna}]
Assume that $\varphi(\DD)\subset D$ and $\varphi$ is a complex geodesic for $D$ and let $h$ be as in Theorem \ref{th_postac_geod}.
The conditions \eqref{crl_wzor_geod_w_c2_if_WD_polplaszczyzna_g_in_pd_h}, \eqref{crl_wzor_geod_w_c2_if_WD_polplaszczyzna_rho_pionowe} and \eqref{crl_wzor_geod_w_c2_if_WD_polplaszczyzna_fi_0} follow immediately from Theorem \ref{th_postac_geod}.
Since $\BLHL{}\in\overline{W_D}$ for every $\lambda\in\TT$, we have $h_2\in-\HHP^1$, what gives \eqref{crl_wzor_geod_w_c2_if_WD_polplaszczyzna_h2_w_minus_Hplus}.

If $h_2\not\equiv 0$, then $h_2$ has at most one root on $\TT$ (counting without multiplicities), so the set $\lbrace\lambda\in\TT:\BLHL{}\in\partial W_D\rbrace$ contains at most one element.
Hence the condition \eqref{crl_wzor_geod_w_c2_if_WD_polplaszczyzna_varrho_nu_if_h2_nie_zero} follows from Theorem \ref{th_postac_geod} \eqref{th_postac_geod_nu_skupione_na}.

It is a direct consequence of Theorem \ref{th_postac_geod} that if $h$ is such that the conditions \eqref{crl_wzor_geod_w_c2_if_WD_polplaszczyzna_h2_w_minus_Hplus} - \eqref{crl_wzor_geod_w_c2_if_WD_polplaszczyzna_varrho_nu_if_h2_nie_zero} are fulfilled, then $\varphi(\DD)\subset D$ and $\varphi$ is a complex geodesic for $D$.
\end{proof}

\begin{EX}\label{ex_polparabola}
Consider the domain $$D:=\lbrace (x_1,x_2)\in\RR^2:x_1>0, x_2>x_1^2\rbrace+i\RR^2.$$
This domain is of the type considered in Corollary \ref{crl_wzor_geod_w_c2_if_WD_polplaszczyzna}, because $$W_D=\lbrace (r\cos\theta,r\sin\theta):r\geq 0,\theta\in[-\pi,0)\rbrace.$$
For $v=(v_1,v_2)\in\RR^2$ we have
$$P_D(v)=\left\lbrace\begin{array}{ll}
\left\lbrace\left(-\frac{v_1}{2v_2}, \frac{v_1^2}{4v_2^2}\right)\right\rbrace, &\text{if }v\in(0,\infty)\times(-\infty,0),\\
\lbrace(0,0)\rbrace, &\text{if }v\in(-\infty,0]\times(-\infty,0),\\
\lbrace 0\rbrace\times[0,\infty), &\text{if }v\in(-\infty,0)\times\lbrace 0\rbrace,\\
\varnothing, &\text{otherwise}.
\end{array}\right.$$

Take a complex geodesic $\varphi:\DD\to D$ with boundary measure $\mu=(\mu_1,\mu_2)$ and let $g=(g_1,g_2)$, $\nu$, $\varrho=(\varrho_1,\varrho_2)$ and $h=(h_1,h_2)\in\HH^1\times(-\HHP^1)$ be as in Corollary \ref{crl_wzor_geod_w_c2_if_WD_polplaszczyzna}.
We have $\mu_1=g_1 d\LEBT$, so from the conditions \eqref{crl_wzor_geod_w_c2_if_WD_polplaszczyzna_g_in_pd_h} and \eqref{crl_wzor_geod_w_c2_if_WD_polplaszczyzna_fi_0} of Corollary \ref{crl_wzor_geod_w_c2_if_WD_polplaszczyzna} it follows that the sets $\lbrace\lambda\in\TT:g_1(\lambda)>0\rbrace$, $\lbrace\lambda\in\TT:\BLHL{1}>0\rbrace$ are of positive $\LEBT$ measure and $h_2\not\equiv 0$.
In particular, $h_1\in\HH^1\setminus(-\HHP^1)$.
The condition \eqref{crl_wzor_geod_w_c2_if_WD_polplaszczyzna_varrho_nu_if_h2_nie_zero} implies that $$\varrho\,d\nu=(0,\alpha\delta_{\lambda_0})$$ for some $\alpha\in[0,\infty)$ and $\lambda_0\in\TT$ such that $\alpha h_2(\lambda_0)=0$.
Moreover, as $h_2\not\equiv 0$, for $\LEBT$-a.e. $\lambda\in\lbrace\zeta\in\TT:\bar\zeta h_1(\zeta)\leq 0\rbrace$ there holds $g(\lambda)=(0,0)$.
This gives $$g(\lambda)=\left(-\frac{h_1(\lambda)}{2h_2(\lambda)}, \frac{h_1(\lambda)^2}{4h_2(\lambda)^2}\right)\chi_{\lbrace\zeta\in\TT:\bar\zeta h_1(\zeta)>0\rbrace}(\lambda)$$ for $\LEBT$-a.e. $\lambda\in\TT$.

If the functions $h_1$, $h_2$ are linearly independent, then $h_2$ has no roots on the set $\lbrace\zeta\in\TT:\bar\zeta h_1(\zeta)\geq 0\rbrace$, because $g_1,g_2\in L^1(\TT,\LEBT)$.
In that case
\begin{equation}\label{eq_expolp_mu_niezal}
\mu=\left(-\frac{h_1(\lambda)}{2h_2(\lambda)}, \frac{h_1(\lambda)^2}{4h_2(\lambda)^2}\right)\chi_{\lbrace\zeta\in\TT:\bar\zeta h_1(\zeta)>0\rbrace}(\lambda)d\LEBT(\lambda)+(0,\alpha\delta_{\lambda_0}).
\end{equation}
On the other hand, if a holomorphic map has boundary measure of the form \eqref{eq_expolp_mu_niezal} with some $h_1\in\HH^1\setminus(-\HHP^1)$, $h_2\in-\HHP^1$, $\alpha\in[0,\infty)$ and $\lambda_0\in\TT$ such that $h_1$, $h_2$ are linearly independent, $h_2$ has no roots on $\lbrace\zeta\in\TT:\bar\zeta h_1(\zeta)\geq 0\rbrace$ and $\alpha h_2(\lambda_0)=0$, then it is a complex geodesic for $D$.

If $h_1$, $h_2$ are linearly dependent, then applying Corollary \ref{crl_wzor_geod_w_c2_if_WD_polplaszczyzna} \eqref{crl_wzor_geod_w_c2_if_WD_polplaszczyzna_varrho_nu_if_h2_nie_zero} and arguing similarly as in previously considered examples we obtain that $$\mu=(\gamma,\gamma^2)\,d\LEBT+(0,\alpha\delta_{\lambda_0})$$ for some $\alpha<0$, $\lambda_0\in\TT$ and $\gamma>0$.
Any holomorphic map with boundary measure of the above form is a complex geodesic for $D$.
\end{EX}

\section{Proof of Theorem \ref{th_postac_geod} and further remarks}\label{sect_dowod_tw_postac_geod}

The aim of this section is to prove the main result of the paper, Theorem \ref{th_postac_geod}.
We begin with investigating the singular and absolutely continuous parts of the boundary measure of a complex geodesic in its Lebesgue-Radon-Nikodym decomposition with respect to $\LEBT$.
Next we show the proof of Theorem \ref{th_postac_geod} and we give some remarks related to it.

A starting point for our considerations is the following fact (see \cite[Theorem 1.2]{zajac}):

\begin{THE}\label{th_wkw_miara}
Let $D\subset\CC^n$ be a convex tube domain containing no complex affine lines and let $\varphi:\DD\to D$ be a holomorphic map with boundary measure $\mu$.
Then $\varphi$ is a complex geodesic for $D$ iff there exists a map $h\in\HH^n$, $h\not\equiv 0$, such that $$\langle\BLHL{},\RE z\,\DLEBT(\lambda)-d\mu(\lambda)\rangle\leq 0$$ for every $z\in D$.
\end{THE}

\begin{LEM}\label{lem_g_wkw_rozklad_lrn}
Let $D\subset\CC^n$ be a convex tube domain containing no complex affine lines, $h\in\HH^n$, $h\not\equiv 0$ and let $\varphi:\DD\to D$ be a holomorphic map with boundary measure $\mu$. 
Consider $$\mu=\RE\varphi^*\,d\LEBT+\mu_s,$$ the Lebesgue-Radon-Nikodym decomposition of $\mu$ with respect to $\LEBT$.
Then
\begin{equation}\label{eq_g_lem_lrn_m}
\langle\BLHL{},\RE z\,\DLEBT(\lambda)-d\mu(\lambda)\rangle\leq 0\text{ for each }z\in D
\end{equation}
iff the following two conditions hold:
\begin{enumerate}[(i)]
\item\label{lem_g_wkw_rozklad_lrn_1} $\RE\varphi^*(\lambda)\in P_D(\BLHL{})$ for $\LEBT$-a.e. $\lambda\in\TT$,
\item\label{lem_g_wkw_rozklad_lrn_2} $\langle\BLHL{}, d\mu_s(\lambda)\rangle\geq 0$.
\end{enumerate}
\end{LEM}

\begin{proof}
Let $\mu_s=(\mu_{s,1},\ldots,\mu_{s,n})$.
There exists a Borel subset $S\subset\TT$ such that $$\LEBT(S)=0,\;|\mu_{s,1}|(\TT\setminus S)=\ldots=|\mu_{s,n}|(\TT\setminus S)=0.$$
There hold the equalities
\begin{equation}\label{eq_g_lwkwrozklad_lrn_sts}
\chi_S\,d\LEBT=0,\; \chi_{\TT\setminus S}\,d\LEBT=\LEBT,\; \chi_S\,d\mu=\mu_s,\; \chi_{\TT\setminus S}\,d\mu=\RE\varphi^*\,d\LEBT.
\end{equation}
For $z\in D$ set
\begin{equation*}
\nu_z := \langle\BLHL{},\RE z\,\DLEBT(\lambda)-d\mu(\lambda)\rangle.
\end{equation*}
We have $\nu_z=\chi_{\TT\setminus S}\,d\nu_z+\chi_S\,d\nu_z$ and from \eqref{eq_g_lwkwrozklad_lrn_sts} it follows that
\begin{equation}\label{eq_g_lwkwrozklad_lrn_chi_ts_nu_z}
\chi_{\TT\setminus S}\,d\nu_z=\langle\BLHL{},\RE z-\RE\varphi^*(\lambda)\rangle\,d\LEBT(\lambda)
\end{equation}
and
\begin{equation}\label{eq_g_lwkwrozklad_lrn_chi_s_nu_z}
\chi_S\,d\nu_z=-\langle\BLHL{},d\mu_s(\lambda)\rangle.
\end{equation}

If the condition \eqref{eq_g_lem_lrn_m} holds, i.e. $\nu_z\leq 0$ for every $z\in D$, then \eqref{lem_g_wkw_rozklad_lrn_1} follows from \cite[Lemma 3.7]{zajac} and \eqref{lem_g_wkw_rozklad_lrn_2} follows from the equality \eqref{eq_g_lwkwrozklad_lrn_chi_s_nu_z}.
On the other hand, if there hold both \eqref{lem_g_wkw_rozklad_lrn_1} and \eqref{lem_g_wkw_rozklad_lrn_2}, then \eqref{eq_g_lwkwrozklad_lrn_chi_ts_nu_z} and \eqref{eq_g_lwkwrozklad_lrn_chi_s_nu_z} gives that for each $z\in D$ the measures $\chi_{\TT\setminus S}\,d\nu_z$ and $\chi_S\,d\nu_z$ are negative and hence $\nu_z$ is so.
\end{proof}

\begin{LEM}\label{lem_fi_obraz_w_D_iff_mu_s_bullet_w_niedodatnia_i_gr_radialne}
Let $D\subset\CC^n$ be a convex tube domain containing no complex affine lines, let $\varphi\in\MMM^n$ be a holomorphic map with boundary measure $\mu$ and let $$\mu=\RE\varphi^*\,d\LEBT+\mu_s$$ be the Lebesgue-Radon-Nikodym decomposition of $\mu$ with respect to $\LEBT$.
Then $\varphi(\DD)\subset\overline{D}$ iff the following two conditions hold:
\begin{enumerate}[(i)]
\item\label{lem_fi_obraz_w_D_iff_mu_s_bullet_w_niedodatnia_i_gr_radialne_re_fi_gw} $\RE\varphi^*(\lambda)\in\overline{\RE D}$ for $\LEBT$-a.e. $\lambda\in\TT$,
\item\label{lem_fi_obraz_w_D_iff_mu_s_bullet_w_niedodatnia_i_gr_radialne_mu_s_w} $\langle\mu_s,w\rangle\leq 0$ for every $w\in\overline{W_D}$.
\end{enumerate}
\end{LEM}

\begin{proof}
Again, let $S\subset\TT$ be such that there holds \eqref{eq_g_lwkwrozklad_lrn_sts}.
Assume that $\varphi(\DD)\subset\overline{D}$.
The first condition is clear.
If $v\in W_D$, then for some constant $C\in\RR$ there is $\langle x,v\rangle<C$ for every $x\in\RE D$.
In particular, $\langle\RE\varphi(\lambda),v\rangle<C$ for $\lambda\in\DD$, what gives a similar inequality for measures:
$$\langle\RE\varphi(r\lambda)\,d\LEBT(\lambda),v\rangle\leq C\,d\LEBT,\;r\in(0,1).$$
Taking limit for $r$ tending to $1$ we get $$\langle d\mu,v\rangle\leq C\,d\LEBT.$$
Hence $$\langle \chi_S\,d\mu,v\rangle\leq C\chi_S\,d\LEBT,$$ what together with \eqref{eq_g_lwkwrozklad_lrn_sts} gives $$\langle \mu_s,v\rangle\leq 0.$$
If $w\in\overline{W_D}$, then there exists a sequence $(v_m)_m\subset W_D$ tending to $w$.
The measure $\langle\mu_s,w\rangle$ is a weak-* limit of the sequence $\langle\mu_s,v_m\rangle$ of negative measures, so it is also negative.

Now assume that both \eqref{lem_fi_obraz_w_D_iff_mu_s_bullet_w_niedodatnia_i_gr_radialne_re_fi_gw} and \eqref{lem_fi_obraz_w_D_iff_mu_s_bullet_w_niedodatnia_i_gr_radialne_mu_s_w} hold.
It suffices to show that if $p\in\RR^n\setminus\overline{\RE D}$ and $v\in\RR^n$ are such that $\langle x-p,v\rangle\leq 0$ for every $x\in\overline{\RE D}$, then $\langle \RE\varphi(\lambda)-p,v\rangle\leq 0$ for every $\lambda\in\DD$.
Fix $p$, $v$ and $\lambda$.
It is clear that $v\in W_D$ and $\langle\RE\varphi^*(\zeta)-p,v\rangle\leq 0$ for $\LEBT$-a.e. $\zeta\in\TT$.
We have
\begin{eqnarray*}
\langle\RE\varphi(\lambda)-p,v\rangle &=& \frac{1}{2\pi}\int_{\TT}\frac{1-|\lambda|^2}{|\zeta-\lambda|^2}\langle\RE\varphi^*(\zeta)-p,v\rangle d\LEBT(\zeta)\\
 &+& \frac{1}{2\pi}\int_{\TT}\frac{1-|\lambda|^2}{|\zeta-\lambda|^2}\,d(\langle\mu_s(\zeta),v\rangle),
\end{eqnarray*}
so $\langle\RE\varphi(\lambda)-p,v\rangle\leq 0$ and the proof is complete.
\end{proof}

\begin{proof}[Proof of Theorem \ref{th_postac_geod}]
We have $$\RE\varphi^*(\lambda)=g(\lambda)\text{ for }\LEBT\text{-a.e. }\lambda\in\TT$$ and $$\mu_s=\varrho\,d\nu,$$ where $\mu_s=(\mu_{s,1},\ldots,\mu_{s,n})$ is the singular part of $\mu$ in its Lebesgue-Radon-Nikodym decomposition with respect to $\LEBT$.
The condition \eqref{lem_g_wkw_rozklad_lrn_2} from Lemma \ref{lem_g_wkw_rozklad_lrn} may be written as
\begin{equation}\label{eq_thpg_warunek_z_lematu_lrn_oznacza}
\langle\BLHL{},\varrho(\lambda)\rangle\,d\nu(\lambda)\geq 0 
\end{equation}
and the condition \eqref{lem_fi_obraz_w_D_iff_mu_s_bullet_w_niedodatnia_i_gr_radialne_mu_s_w} from Lemma \ref{lem_fi_obraz_w_D_iff_mu_s_bullet_w_niedodatnia_i_gr_radialne} can be written as
\begin{equation}\label{eq_thpg_warunek_z_lematu_o_obrazie_oznacza}
\langle\varrho(\lambda),w\rangle\,d\nu(\lambda)\leq 0\text{ for every }w\in\overline{W_D}.
\end{equation}
Now it is clear that if for some map $h\in\HH^n$, $h\not\equiv 0$ the conditions \eqref{th_postac_geod_g_in_pd_h} - \eqref{th_postac_geod_fi_0} from Theorem \ref{th_postac_geod} hold, then from Lemmas \ref{lem_g_wkw_rozklad_lrn} and \ref{lem_fi_obraz_w_D_iff_mu_s_bullet_w_niedodatnia_i_gr_radialne} it follows that $\varphi(\DD)\subset D$ and $\varphi$ is a complex geodesic for $D$.
It remains to prove the opposite implication.

Assume that $\varphi(\DD)\subset D$ and $\varphi$ is a complex geodesic for $D$.
Take $h\in\HH^n$ as in Theorem \ref{th_wkw_miara}.
The condition \eqref{th_postac_geod_fi_0} is clear and the conditions \eqref{th_postac_geod_g_in_pd_h}, \eqref{th_postac_geod_h_bullet_rho_wieksze} of Theorem \ref{th_postac_geod} follow directly from \eqref{eq_thpg_warunek_z_lematu_lrn_oznacza} and Lemma \ref{lem_g_wkw_rozklad_lrn}.

Lemma \ref{lem_fi_obraz_w_D_iff_mu_s_bullet_w_niedodatnia_i_gr_radialne} and the equality \eqref{eq_thpg_warunek_z_lematu_o_obrazie_oznacza} imply that for every $w\in W_D$ and $\nu$-a.e. $\lambda\in\TT$ there holds
\begin{equation}\label{eq_thpg_F_lambda_bullet_w}
\langle\varrho(\lambda),w\rangle\leq 0.
\end{equation}
This 'almost every' may a priori depend on $w$, but we can omit this problem in the following way.
Take a dense, countable subset $\lbrace w_j:j=1,2,\ldots\rbrace\subset W_D$ and for each $j$ let $A_j\subset\TT$ be a Borel set such that $\nu(\TT\setminus A_j)=0$ and $\langle\varrho(\lambda),w_j\rangle\leq 0$ for every $\lambda\in A_j$.
Put $A:=\cap_{j=1}^\infty A_j$.
It is clear that $\nu(\TT\setminus A)=0$ and \eqref{eq_thpg_F_lambda_bullet_w} holds every $w\in W_D$ and every $\lambda\in A$.
Thus,
\begin{equation*}
\varrho(\lambda)\in S_D\text{ for }\nu\text{-a.e. }\lambda\in\TT.
\end{equation*}
This is exactly the condition \eqref{th_postac_geod_rho_w_SD}.

It remains to prove the last part of the theorem, i.e. if $h\in\HH^n$, $h\not\equiv 0$ satisfy the conditions \eqref{th_postac_geod_g_in_pd_h} - \eqref{th_postac_geod_fi_0}, then it satisfy also \eqref{th_postac_geod_h_bullet_rho_rowne_0}, \eqref{th_postac_geod_nu_skupione_na} and \eqref{th_postac_geod_blhl_in_WD}.

From \eqref{th_postac_geod_g_in_pd_h} it follows that for $\LEBT$-a.e. $\lambda\in\TT$ there is $\BLHL{}\in W_D$.
Hence, \eqref{th_postac_geod_blhl_in_WD} is a consequence of continuity of $h$.

We prove \eqref{th_postac_geod_h_bullet_rho_rowne_0}.
Fix $\epsilon>0$.
There exists $\delta>0$ such that $|\bar\lambda h_j(\lambda)-\bar\zeta h_j(\zeta)|\leq\epsilon$ for $j=1,\ldots,n$, whenever $\lambda,\zeta\in\TT$ and $|\lambda-\zeta|\leq\delta$.
Take $\lambda_1,\ldots,\lambda_m\in\TT$ for which the arcs $L_k:=\lbrace \lambda\in\TT:|\lambda-\lambda_k|<\delta\rbrace$, $k=1,\ldots,m$, cover the circle $\TT$.
For $\nu$-a.e. $\lambda\in L_k$ we have $$\langle\BLHL{},\varrho(\lambda)\rangle\leq \langle\bar\lambda_k h(\lambda_k),\varrho(\lambda)\rangle+\|\bar\lambda h(\lambda)-\bar\lambda_k h(\lambda_k)\|\leq\epsilon\sqrt{n}.$$
The last inequality follows from \eqref{th_postac_geod_rho_w_SD} and \eqref{th_postac_geod_blhl_in_WD}.
As $k$ and $\epsilon$ are arbitrary, the condition \eqref{th_postac_geod_h_bullet_rho_rowne_0} follows.

Now we prove \eqref{th_postac_geod_nu_skupione_na}.
For every $\lambda\in\TT$ such that $\varrho(\lambda)\in S_D$ and $\BLHL{}\in\INT W_D$ there holds $\langle\BLHL{},\varrho(\lambda)\rangle<0$, because the map $w\mapsto\langle\varrho(\lambda),w\rangle$ is open and non-positive on $W_D$, so it must be negative on $\INT W_D$.
Hence, in view of \eqref{th_postac_geod_h_bullet_rho_rowne_0}, $\BLHL{}\in\INT W_D$ holds $\nu$-almost nowhere on $\TT$.
The proof is complete.
\end{proof}

\begin{RM}
It follows from the prooof, that if $\varphi$ is a complex geodesic for $D$ and $h$ is as in Theorem \ref{th_wkw_miara}, then all of the conditions from Theorem \ref{th_postac_geod} are satisfied with this $h$.
And vice versa, if the conditions \eqref{th_postac_geod_g_in_pd_h} - \eqref{th_postac_geod_fi_0} from Theorem \ref{th_postac_geod} hold for $h\in\HH^n$, $h\not\equiv 0$, then $h$ satisfy the condition from Theorem \ref{th_wkw_miara}.
\end{RM}

\begin{EX}\label{ex_obszar_polstozek}
In most of previously considered domains the singular part of boundary measure of a complex geodesic took a very special form (it was expressed by Dirac deltas), provided that the components of corresponding map $h$ were linearly independent.
In this example we will see that in some domains even for such $h$ the singular part may be almost arbitrary.
Consider tube domain in $\CC^2$ with the base being a 'half-cone', namely $$D:=\left\lbrace (x_1,x_2,x_3)\in\RR^3: x_2>0, x_3>\sqrt{x_1^2+x_2^2}\right\rbrace+i\RR^3.$$
One can check that $$W_D=\left\lbrace x_1\in\RR, x_2\geq 0, x_3\leq -\sqrt{x_1^2+x_2^2}\right\rbrace\cup\left\lbrace x_1\in\RR, x_2\leq 0,x_3\leq-|x_1|\right\rbrace$$ and $$S_D=\overline{\RE D}.$$

Let $h\in\HH^2$ be such that $$\BLHL{}=(\RE\lambda,\IM\lambda,-1),\;\lambda\in\TT,$$ i.e. $h(\lambda)=\frac12\left(\lambda^2+1,-i\lambda^2+i,-2\lambda\right)$.
Set $g:=0$ and $$\varrho(\lambda):=2^{-\frac 12}(\RE\lambda,\IM\lambda,1),\;\lambda\in\TT.$$
Note that for $\lambda\in\TT$ there holds $\BLHL{}\in\partial W_D$ iff $\IM\lambda\geq 0$.
%, so in opposite to previously considered domains, in this example the set $\lbrace\lambda\in\TT:\BLHL{}\in\partial W_D\rbrace$ is neither the whole $\TT$ nor a finite subset of $\TT$.
Let $\nu$ be an arbitrary finite positive Borel measure on $\TT$ singular to $\LEBT$ and such that $$\nu(\lbrace\lambda\in\TT:\IM\lambda<0\rbrace)=0.$$
Set $\mu:=g\,d\LEBT+\varrho\,d\nu=\varrho\,d\nu$ and let $\varphi$ be a holomorphic map given by the boundary measure $\mu$.
One can see that the conditions \eqref{th_postac_geod_g_in_pd_h}, \eqref{th_postac_geod_h_bullet_rho_wieksze} and \eqref{th_postac_geod_rho_w_SD} from Theorem \ref{th_postac_geod} are fulfilled.

Now if we choose $\nu$ such that $\frac{1}{2\pi}\mu(\TT)\in\RE D$, then in view of Theorem \ref{th_postac_geod} the map $\varphi$ is a complex geodesic for $D$.
To do so, we can e.g. take an arbitrary finite positive Borel measure $\omega$ singular to $\LEBT$ and supported on the set $\lbrace\lambda\in\TT:\IM\lambda\geq 0\rbrace$, and put $\nu:=\omega+\delta_1+\delta_i$.
\end{EX}

\begin{RM}\label{rm_sprowadzenie_do_wymiaru_1_2_3}
Let $D\subset\CC^n$ be a convex tube domain containing no complex affine lines.
% In this remark we shall see that the question on finding all complex geodesics for $D$ may be in some sense reduced to question on finding complex geodesics for some family of convex tube domins of dimension at most three.
% More precisely, a
Then a map $\varphi\in\OO(\DD,D)$ is a complex geodesic for $D$ iff there exists a number $m\in\lbrace 1,2,3\rbrace$ and a real $m\times n$ matrix $V$ with linearly independent rows such that the domain $D':=\lbrace V\cdot z:z\in D\rbrace\subset\CC^m$ is a convex tube containing no complex affine lines and $V\cdot\varphi$ is a complex geodesic for $D'$.
% This claim follows from \cite[Lemma 4.3]{zajac}: if $\varphi$ is a complex geodesic for $D$ and $h(\lambda)=\bar a\lambda^2+b\lambda+a$ ($a\in\CC^n$, $b\in\RR^n$) is as in Theorem \ref{th_postac_geod}, then $m$ may be chosen as the maximal number of linearly independent functions among $h_1,\ldots,h_n$ and $V$ may be chosen such that its rows form a basis of the space $X_h:=\SPAN_{\RR}\lbrace\RE a,\IM a,b\rbrace$.
This claim follows from \cite[Lemma 4.3]{zajac}: if $\varphi$ is a complex geodesic for $D$ and $h(\lambda)=\bar a\lambda^2+b\lambda+a$ ($a\in\CC^n$, $b\in\RR^n$) is as in Theorem \ref{th_postac_geod}, then $V$ may be chosen such that its rows form a basis of the space $X_h:=\SPAN_{\RR}\lbrace\RE a,\IM a,b\rbrace$.
Moreover, if in this situation we do an affine change of coordinates so that $X_h=\RR^m\times\lbrace 0\rbrace^{n-m}$, then we conclude that the map $(\varphi_1,\ldots,\varphi_m)$ has to be a complex geodesic for $D'$ and the other components $\varphi_{m+1},\ldots,\varphi_n$ can be arbitrary, privided that $\varphi(\DD)\subset D$.
\end{RM}

\section{Applications of Theorem \ref{th_postac_geod} in Reinhardt domains in $\CC^n$}\label{sect_reinhardt_domains}

In this section we apply results obtained in Subsection \ref{subsect_domains_Dn} to give formulas (more precisely, a necessary condition) for extremal mappings with respect to the Lempert function and the Kobayashi-Royden pseudometric in some classes of complete Reinhardt domains in $\CC^n$.
Recall that a non-empty open set $G\subset\CC^n$ is called a \emph{complete Reinhardt domain} if for every $(z_1,\ldots,z_n)\in G$ and $\lambda_1,\ldots,\lambda_n\in\CDD$ there holds $(\lambda_1 z_1,\ldots,\lambda_n z_n)\in G$.
To such a domain $G\subset\CC^n$ we associate its \emph{logarithmic image} $$\log G:=\lbrace (\log|z_1|,\ldots,\log |z_n|)\in\RR^n:(z_1,\ldots,z_n)\in G\cap(\CC_*)^n\rbrace$$ and the tube domain $$D_G:=\log G+i\RR^n$$ (we have $\RE D_G=\log G$).
The map $$\exp:D_G\ni (z_1,\ldots,z_n)\mapsto(e^{z_1},\ldots,e^{z_n})\in G\cap(\CC_*)^n$$ is then a holomorphic covering.
If the domain $G$ is bounded and pseudoconvex, then $D_G$ belongs to the family $\mathcal{D}_n$.
Using an argument from \cite{edigarianzwonek} we obtain a relation between extremal mappings in $G$ and complex geodesics in $D_G$.
It allows us to apply our description of complex geodesics for domain $D_G\in\mathcal{D}_n$ (Corollary \ref{crl_postac_geod_Dn}) to get formulas for extremal mappings in $G$.

Let $D\subset\CC^n$ be a domain.
The \emph{Lempert function} $\ell_D:D\times D\to [0,\infty)$ for $D$ is given by
\begin{eqnarray*}
\ell_D(z,w) \! &=& \! \inf\left\lbrace \rho(\sigma_1,\sigma_2):\sigma_1,\sigma_2\in\DD,\exists f\in\OO(\DD,D):f(\sigma_1)=z,f(\sigma_2)=w\right\rbrace\\
 \! &=& \! \inf\left\lbrace \rho(0,\sigma):\sigma\in\DD,\exists f\in\OO(\DD,D):f(0)=z,f(\sigma)=w\right\rbrace
\end{eqnarray*}
and the Kobayashi-Royden pseudometric $\kappa_D:D\times\CC^n\to[0,\infty)$ for $D$ is
\begin{eqnarray*}
\kappa_D(z,X) \!\!\! &=& \!\!\! \inf\left\lbrace \tfrac{|\alpha|}{1-|\sigma|^2}:\alpha\in\CC,\sigma\in\DD, \exists f\in\OO(\DD,D):f(\sigma)=z,\alpha f'(\sigma)=X\right\rbrace\\
 \!\!\! &=& \!\!\! \inf\left\lbrace \alpha>0:\exists f\in\OO(\DD,D):f(0)=z,\alpha f'(0)=X\right\rbrace.
\end{eqnarray*}
We say that a holomorphic map $f:\DD\to D$ is a $\ell_D$-\emph{extremal map} if $\rho(\sigma_1,\sigma_2)=\ell_D(f(\sigma_1),f(\sigma_2))$ for some $\sigma_1,\sigma_2\in\DD$ such that $\sigma_1\neq\sigma_2$.
We call $f$ a $\kappa_D$-\emph{extremal map} if $\kappa_D(f(\sigma),f'(\sigma))=\frac{1}{1-|\sigma|^2}$ for some $\sigma\in\DD$.
We often use the following basic fact: given $\sigma_1,\sigma_2\in\DD$, $\sigma_1\neq\sigma_2$ and $f\in\OO(\DD,D)$, the equality $\rho(\sigma_1,\sigma_2)=\ell_D(f(\sigma_1),f(\sigma_2))$ holds iff there is no map $g\in\OO(\DD,D)$ such that $g(\sigma_1)=f(\sigma_1)$, $g(\sigma_2)=f(\sigma_2)$ and $g(\DD)\subset\subset D$.
And analogously, given $\sigma\in\DD$ and $f\in\OO(\DD,D)$, the equality $\kappa_D(f(\sigma),f'(\sigma))=\frac{1}{1-|\sigma|^2}$ holds iff there is no map $g\in\OO(\DD,D)$ such that $g(\sigma)=f(\sigma)$, $g'(\sigma)=f'(\sigma)$ and $g(\DD)\subset\subset D$.
% If $D$ is a taut convex domain, then by the Lempert theorem all $\ell_D$-extremal and $\kappa_D$-extremal maps are complex geodesics.

For a convex tube domain $D\subset\CC^n$ containing no complex affine lines let $\mathcal{G}(D)$ denote the family of all Borel-measurable maps $g=(g_1,\ldots,g_n):\TT\to\RR^n$ such that $g_1,\ldots,g_n\in L^1(\TT,\LEBT)$ and with some $h\in\HH^n$ there holds $$g(\lambda)\in P_D(\BLHL{})\text{ for }\LEBT\text{-a.e. }\lambda\in\TT.$$
It follows from Theorem \ref{th_postac_geod} that if $g\in\mathcal{G}(D)$ and $\varphi\in\mathcal{M}^n$ is a map with boundary measure $g\,d\LEBT$, then either $\varphi$ is a complex geodesic for $D$ (if $\varphi(0)\in D$) or its image lies in $\partial D$ (in the opposite situation).
Note that if $\varphi(\DD)\subset\partial D$ and in addition the domain $\RE D$ is strictly convex in the geometric sense, then $\varphi$ is just a constant map.

In what follows, for a non-empty set $A=\lbrace j_1,\ldots,j_k\rbrace\subset\lbrace 1,\ldots,n\rbrace$ with $j_1<\ldots<j_k$, by $\pi_A$ we denote the projection $\CC^n\to\CC^k$ on the coordinates $j_1,\ldots,j_k$.

In the following two propositions we present formulas for $\ell_G$-extremal and $\kappa_G$-extremal maps in some classes of bounded, pseudoconvex, complete Reinhardt domains.

\begin{PROP}\label{prop_extremals_in_reinhardt_in_Cn_D_G_str_conv}
Let $G\subset\CC^n$, $n\geq 2$, be a bounded, pseudoconvex, complete Reinhardt domain such that the domain $\log G$ is strictly convex in the geometric sense and let $f=(f_1,\ldots,f_n)\in\OO(\DD,G)$ be a $\ell_G$-extremal or a $\kappa_G$-extremal map.
Set $$A:=\lbrace j\in\lbrace 1,\ldots,n\rbrace:f_j\not\equiv 0\rbrace$$ and let $k$ denote the number of elements of $A$.

Then $k>0$ and there exist some functions $B_1,\ldots,B_k\in\AUT(\DD)\cup\lbrace 1\rbrace$ and a map $g\in\mathcal{G}(D_{\pi_A(G)})$ such that
\begin{equation*}
\pi_A\circ f=(B_1 e^{\varphi_1},\ldots,B_k e^{\varphi_k}),
\end{equation*}
where $\varphi=(\varphi_1,\ldots,\varphi_k)\in\mathcal{M}^k$ is a map with boundary measure $g\,d\LEBT$.
\end{PROP}

\begin{PROP}\label{prop_extremals_in_reinhardt_in_C2}
Let $G\subset\CC^2$ be a bounded, pseudoconvex, complete Reinhardt domain and let $R_1,R_2>0$ be such that $\pi_{\lbrace 1\rbrace}(G)=R_1\cdot\DD$ and $\pi_{\lbrace 2\rbrace}(G)=R_2\cdot\DD$.
If $f\in\OO(\DD,G)$ is a $\ell_G$-extremal or a $\kappa_G$-extremal map, then there holds \textnormal{at least one} of the following conditions:
\begin{enumerate}[(i)]
\item\label{prop_extremals_in_reinhardt_in_C2_f_j_aut} there exists $j\in\lbrace 1,2\rbrace$ such that $\frac{1}{R_j}f_j\in\AUT(\DD)$, \textnormal{or}
\item\label{prop_extremals_in_reinhardt_in_C2__f_j_nie_aut} there exist some $B_1,B_2\in\AUT(\DD)\cup\lbrace 1\rbrace$ and $g\in\mathcal{G}(D_G)$ such that $f$ is of the form
\begin{equation*}
f=(B_1 e^{\varphi_1},B_2 e^{\varphi_2}),
\end{equation*}
where $\varphi=(\varphi_1,\varphi_2)\in\mathcal{M}^2$ is a map with boundary measure $g\,d\LEBT$.
\end{enumerate}
\end{PROP}

\begin{proof}[Proof of Proposition \ref{prop_extremals_in_reinhardt_in_Cn_D_G_str_conv}]
We present the proof only for the case when $f$ is a $\ell_G$-extremal map, because the proof for $\kappa_G$-extremal map is analogous.
Let $\sigma_1,\sigma_2\in\DD$ be such that $\rho(\sigma_1,\sigma_2)=\ell_G(f(\sigma_1),f(\sigma_2))$ and $\sigma_1\neq\sigma_2$.
It is clear that $k>0$.
The domain $\pi_A(G)$ satisfies the same assumptions as $G$, namely it is a bounded, pseudoconvex, complete Reinhardt domain with $\RE D_{\pi_A(G)}$ being strictly convex.
Moreover, if $z=(z_1,\ldots,z_n),w=(w_1,\ldots,w_n)\in G$ are such that $z_j=w_j=0$ for every $j\not\in A$, then $\ell_G(z,w)=\ell_{\pi_A(G)}(\pi_A(z),\pi_A(w))$.
In particular, $$\ell_{\pi_A(G)}(\pi_A(f(\sigma_1)),\pi_A(f(\sigma_2)))=\ell_G(f(\sigma_1),f(\sigma_2))=\rho(\sigma_1,\sigma_2),$$ what means that $\pi_A\circ f$ is a $\ell_{\pi_A(G)}$-extremal map.
Therefore, we need only to prove the conclusion for the domain $\pi_A(G)$ and the mapping $\pi_A\circ f$.
The latter map has no components equal identically to zero, so in fact it is enough to prove the proposition under the additional assumption that $f_1,\ldots,f_n\not\equiv 0$ and $A=\lbrace 1,\ldots,n\rbrace$.

Since $f_j$ is bounded and $f_j\not\equiv 0$ for every $j$, we may write (see \cite[p. 76]{koosis})
\begin{equation}\label{eq_percndgsc_f_j_eq_B_j_exp_u_j_psi_j}
f=(B_1 e^{u_1+\psi_1}, \ldots, B_n e^{u_n+\psi_n})
\end{equation}
for a (possibly infinite or identically equal to $1$) Blaschke product $B_j$, a function $u_j\in\mathcal{M}^1$ with boundary measure of the form $\log|f_j^*|\,d\LEBT$ (note that the function $\log|f_j^*|$ belongs to $L^1(\TT,\LEBT)$) and a function $\psi_j\in\mathcal{M}^1$ with $\IM\psi_j(0)=0$ and with boundary measure being finite, negative and singular to $\LEBT$.
Set $\varphi:=(u_1+\psi_1,\ldots,u_n+\psi_n)$.
For every $j$ we have $\RE\psi_j\leq 0$ on $\DD$ and, for $\LEBT$-a.e. $\lambda\in\TT$, $$\RE\psi_j^*(\lambda)=0,\;\RE\varphi_j^*(\lambda)=\RE u_j^*(\lambda),\;|B_j^*(\lambda)|=1,\;|f_j^*(\lambda)|=e^{\RE u_j^*(\lambda)}.$$
In particular, $\varphi(\DD)\subset\overline{D_G}$.

We claim that either the map $\varphi$ is a complex geodesic for $D_G$ or the image of $\varphi$ lies in $\partial D_G$.
The idea of this claim comes from \cite{edigarianzwonek}.
Assume that $\varphi(\DD)\not\subset\partial D_G$, then clearly $\varphi(\DD)\subset D_G$.
If $\varphi$ is not a complex geodesic for $D_G$, then there exists a map $\widetilde{\varphi}=(\widetilde{\varphi}_1,\ldots,\widetilde{\varphi}_n)\in\OO(\DD,D_G)$ such that $\widetilde{\varphi}(\sigma_1)=\varphi(\sigma_1)$, $\widetilde{\varphi}(\sigma_2)=\varphi(\sigma_2)$ and $\widetilde{\varphi}(\DD)\subset\subset D_G$.
Now the map $(B_1 e^{\widetilde{\varphi}_1},\ldots,B_n e^{\widetilde{\varphi}_n})$ maps $\sigma_1$, $\sigma_2$ to $f(\sigma_1)$, $f(\sigma_2)$ and its image is relatively compact in $G$.
It is a contradiction with the equality $\ell_G(f(\sigma_1),f(\sigma_2))=\rho(\sigma_1,\sigma_2)$.

From the above claim we conclude that for $\LEBT$-a.e. $\lambda\in\TT$ there holds $\varphi^*(\lambda)\in\partial D_G$ and hence $f^*(\lambda)\in\partial G$.
Note that from the above considerations it follows that the latter condition holds for every $\ell_G$-extremal map - we will use this fact several times.

We are going to show that
\begin{equation}\label{eq_percndgsc_B_j_e_psi_j_aut_or_1}
B_1 e^{\psi_1},\ldots,B_n e^{\psi_n}\in\AUT(\DD)\cup\lbrace 1\rbrace.
\end{equation}
For this, suppose to the contrary that $B_j e^{\psi_j}\not\in\AUT(\DD)\cup\lbrace 1\rbrace$ for some $j$.
We may assume that $j=1$.
Then $B_1 e^{\psi_1}\in\OO(\DD,\DD)\setminus\AUT(\DD)$, so there exists a function $\xi\in\OO(\DD,\DD)$ such that $\xi(\sigma_1)=B_1(\sigma_1)e^{\psi_1(\sigma_1)}$, $\xi(\sigma_2)=B_1(\sigma_2)e^{\psi_1(\sigma_2)}$ and $\xi(\DD)\subset\subset\DD$.
The map $$F:=(\xi e^{u_1}, f_2, \ldots, f_n)\in\OO(\DD,G)$$ maps $\sigma_1$, $\sigma_2$ to $f(\sigma_1)$, $f(\sigma_2)$, so it is also a $\ell_G$-extremal map.
In particular, $F^*(\lambda)\in\partial G$ for $\LEBT$-a.e. $\lambda\in\TT$.
Observe that $0<|\xi^*(\lambda)e^{u_1^*(\lambda)}|<|f_1^*(\lambda)|$.
Hence, from the fact that $D_G\in\mathcal{D}_n$ and $\LEBT$-almost every $f^*(\lambda)$ belongs to $\partial G$ we conclude that $\partial \RE D_G$ contains a non-trivial segment parallel to the vector $e_1$.
This contradicts strict convexity of $\RE D_G$.

From \eqref{eq_percndgsc_B_j_e_psi_j_aut_or_1} it follows that $B_j\in\AUT(\DD)\cup\lbrace 1\rbrace$ and $\RE\psi_j\equiv 0$ for every $j$.
As $\IM\psi_j(0)=0$, we get $\psi_j\equiv 0$.
Set $g_j:=\log|f_j^*|$ and $g:=(g_1,\ldots,g_n)$.
The boundary measure of $\varphi$ is equal to $g\,d\LEBT$, so to complete the proof we need only to show that $g\in\mathcal{G}(D_G)$.
If $\varphi(\DD)\subset D_G$, then $\varphi$ is a complex geodesic for $D_G\in\mathcal{D}_n$ and the conclusion follows directly from Corollary \ref{crl_postac_geod_Dn}.
In the opposite case, when $\varphi(\DD)\subset\partial D_G$, the map $\varphi$ is constant because of strict convexity of $\RE D_G$.
Thus, the map $g=\RE\varphi^*$ is also constant (up to a set of $\LEBT$ measure zero) and its image lies in $\partial\RE D_G$, so it belongs to $\mathcal{G}(D_G)$.
\end{proof}

\begin{proof}[Proof of Proposition \ref{prop_extremals_in_reinhardt_in_C2}]
We again consider only the case when $f=(f_1,f_2)$ is a $\ell_G$-extremal map.
Take $\sigma_1,\sigma_2\in\DD$ such that $\rho(\sigma_1,\sigma_2)=\ell_G(f(\sigma_1),f(\sigma_2))$ and $\sigma_1\neq\sigma_2$.
If $f_1\equiv 0$ or $f_2\equiv 0$, then similarly as in the previous proof we can show that $f_2=\pi_{\lbrace 2\rbrace}\circ f$ is a $\ell_{R_2\cdot\DD}$-extremal map or $f_1=\pi_{\lbrace 1\rbrace}\circ f$ is a $\ell_{R_1\cdot\DD}$-extremal map.
Then the condition \eqref{prop_extremals_in_reinhardt_in_C2_f_j_aut} is satisfied.
Thus, it remains to consider the situation when $f_1,f_2\not\equiv 0$.
In that case there hold \eqref{eq_percndgsc_f_j_eq_B_j_exp_u_j_psi_j} with $B_1$, $B_2$, $u_1$, $u_2$, $\psi_1$, $\psi_2$ and $\varphi$ as in the previous proof.
Like there, either $\varphi(\DD)\subset\partial D_G$ or $\varphi$ is a complex geodesic for $D_G$, what allows us to conclude that $\varphi^*(\lambda)\in\partial D_G$ and $f^*(\lambda)\in\partial G$ for $\LEBT$-a.e. $\lambda\in\TT$.

We claim that there hold any of the condition \eqref{prop_extremals_in_reinhardt_in_C2_f_j_aut} from Proposition \ref{prop_extremals_in_reinhardt_in_C2} or the condition \eqref{eq_percndgsc_B_j_e_psi_j_aut_or_1} with $n=2$, i.e.
\begin{equation}\label{eq_perc2_B_j_e_psi_j_aut_or_1}
B_1 e^{\psi_1},B_2 e^{\psi_2}\in\AUT(\DD)\cup\lbrace 1\rbrace
\end{equation}
(cf. \cite[Lemat 4.3.3]{klisdoktorat}).
Suppose that $B_1 e^{\psi_1}\not\in\AUT(\DD)\cup\lbrace 1\rbrace$.
There exists a function $\xi\in\OO(\DD,\DD)$ such that $\xi(\sigma_1)=B_1(\sigma_1)e^{\psi_1(\sigma_1)}$, $\xi(\sigma_2)=B_1(\sigma_2)e^{\psi_1(\sigma_2)}$ and $\xi(\DD)\subset\subset\DD$.
Consider the map $$F:=(F_1,f_2):=(\xi e^{u_1},f_2).$$
Like previously, $F$ is a $\ell_G$-extremal map and $F^*(\lambda)\in\partial G$ for $\LEBT$-a.e. $\lambda\in\TT$.
Since for $\LEBT$-a.e. $\lambda\in\TT$ we have $f^*(\lambda)\in\partial G$ and $|F_1^*(\lambda)|<|f_1^*(\lambda)|$, the fact that $D_G\in\mathcal{D}_2$ imply that $$f_2^*(\lambda)\in\partial\pi_{\lbrace 2\rbrace}(G)=R_2\cdot\TT.$$
Put $R:=\sup_{\lambda\in\DD}|f_1(\lambda)|$.
As $G$ is a complete Reinhardt domain, for $\LEBT$-a.e. $\lambda\in\TT$ the bidisc $(|f_1^*(\lambda)|\cdot\DD)\times(|f_2^*(\lambda)|\cdot\DD)$ lies in $G$.
Therefore, the bidisc $(R\cdot\DD)\times(R_2\cdot\DD)$ also lies in $G$ (here we rely on the assumption that $n=2$).
We have $F(\DD)\subset (R\cdot\DD)\times(R_2\cdot\DD)$, so $F$ is a $\ell_{(R\cdot\DD)\times(R_2\cdot\DD)}$-extremal map.
Hence either $\frac{1}{R_2}f_2\in\AUT(\DD)$ or $\frac{1}{R}F_1\in\AUT(\DD)$.
But the image of the latter function is relatively compact in $\DD$, so there must hold $\frac{1}{R_2}f_2\in\AUT(\DD)$.
Applying the same reasoning we can show that if $B_2 e^{\psi_2}\not\in\AUT(\DD)\cup\lbrace 1\rbrace$, then $\frac{1}{R_1}f_1\in\AUT(\DD)$.
It means that at least of the conditions \eqref{eq_perc2_B_j_e_psi_j_aut_or_1} or \eqref{prop_extremals_in_reinhardt_in_C2_f_j_aut} holds.

To complete the proof it suffices to prove that the condition \eqref{prop_extremals_in_reinhardt_in_C2__f_j_nie_aut} follows from \eqref{eq_perc2_B_j_e_psi_j_aut_or_1}.
For this, assume that \eqref{eq_perc2_B_j_e_psi_j_aut_or_1} holds.
As before, we need only to show that the map $g:=(\log|f_1^*|,\log|f_2^*|)=\RE\varphi^*$ belongs to $\mathcal{G}(D_G)$.
If $\varphi(\DD)\subset G_D$, then the conclusion follows from Corollary \ref{crl_postac_geod_Dn}.
In the opposite case, when $\varphi(\DD)\subset\partial D_G$, take a vector $v\in\RR^n$ such that $\langle x-\RE\varphi(0),v\rangle<0$ for every $x\in\RE D_G$.
From the maximum principle for harmonic functions it follows that $\langle\RE\varphi-\RE\varphi(0),v\rangle\equiv 0$.
Defining $h(\lambda):=\lambda\cdot v$ we get $g(\lambda)=\RE\varphi^*(\lambda)\in P_D(\BLHL{})$ for $\LEBT$-a.e. $\lambda\in\TT$.
The proof is complete.
\end{proof}

\begin{EX}\label{ex_reinhardt_klis_domain}
(cf. \cite[Theorem 4.1.4]{klisdoktorat})
Given numbers $p,q\in(0,\infty)$ and $a\in (0,1)$, consider the domain $$G:=G_{a,p,q}:=\lbrace (z_1,z_2)\in\CC^2: |z_1|,|z_2|<1, |z_1|^p|z_2|^q<a\rbrace.$$
It is a bounded, pseudoconvex, complete Reinhardt domain in $\CC^2$ with $$D_{G}=\lbrace (x_1,x_2)\in \RR^2: x_1,x_2<0,\;p x_1+q x_2<\log a\rbrace+i\RR^2.$$

Take $g=(g_1,g_2)\in\mathcal{G}(D_G)$.
Observe that up to a set of $\LEBT$ measure zero there hold any of $g_1\equiv 0$, $g_2\equiv 0$ or $pg_1+qg_2\equiv\log a$.
Indeed, take $h=(h_1,h_2)$ as in the definition of the family $\mathcal{G}(D_G)$.
If $h_1\equiv 0$ or $h_2\equiv 0$, then $g_2\equiv 0$ or $g_1\equiv 0$, respectively.
In the opposite case, for $\LEBT$-a.e. $\lambda\in\TT$ we have $$g(\lambda)\in P_{D_G}(\BLHL{})=\left\lbrace\left(tp^{-1}\log a,(1-t)q^{-1}\log a\right):t\in[0,1]\right\rbrace,$$ so $\LEBT$-almost everywhere on $\TT$ there holds $pg_1+qg_2=\log a$.
From this observation it follows that if $\varphi=(\varphi_1,\varphi_2)$ is a holomorphic map with boundary measure $g\,d\LEBT$, then there holds any of $\RE\varphi_1\equiv 0$, $\RE\varphi_2\equiv 0$ or $p\,\RE\varphi_1+q\,\RE\varphi_2\equiv\log a$.
In the case when $\RE\varphi_1\not\equiv 0$ and $\RE\varphi_2\not\equiv 0$ we therefore get $$\varphi(\lambda)=\left(\psi(\lambda) p^{-1}\log a,(1-\psi(\lambda))q^{-1}\log a+i\beta\right),\;\lambda\in\DD$$ for a number $\beta\in\RR$ and a holomorphic map $\psi:\DD\to\SS$, where $$\SS:=\lbrace\zeta\in\CC:0<\RE\zeta<1\rbrace.$$

From these considerations and from Proposition \ref{prop_extremals_in_reinhardt_in_C2} we conclude that if $f=(f_1,f_2)\in\OO(\DD,G_{a,p,q})$ is a $\ell_{G_{a,p,q}}$-extremal or a $\kappa_{G_{a,p,q}}$-extremal map, then one of following conditions hold:
\begin{enumerate}[(i)]
\item $f_1\in\AUT(\DD)$, \textit{or}
\item $f_2\in\AUT(\DD)$, \textit{or}
\item $f$ is of the form $$f=\left(B_1 \exp\left(\psi p^{-1}\log a\right), B_2 \exp\left((1-\psi)q^{-1}\log a+i\beta\right)\right)$$ for some $\psi\in\OO(\DD,\SS)$, $\beta\in\RR$ and $B_1,B_2\in\AUT(\DD)\cup\lbrace 1\rbrace$ with $B_1 B_2\not\equiv 1$.
\end{enumerate}
\end{EX}

% **** **** **** **** **** **** **** **** **** **** **** **** **** **** **** ****

\bigskip
\textbf{Acknowledgements.} I would like to thank \L{}ukasz Kosi\'{n}ski for bringing my attention to some important papers and for many comments that improved the final shape of the paper.

\end{document}